\documentclass[a4paper,10pt]{amsart}

\usepackage{graphicx}
\usepackage{hyperref}

\usepackage{tikz}

\usepackage{a4wide}

\usepackage{amsmath}
\usepackage{amsfonts}
\usepackage{amssymb}
\usepackage{amsthm}

\usepackage{cite}

\newtheorem{lem}{Lemma}[section]
\newtheorem{thm}[lem]{Theorem}
\newtheorem{prop}[lem]{Proposition}
\newtheorem{cor}[lem]{Corollary}

\numberwithin{equation}{section}

\newtheorem*{cor*}{Corollary}
\newtheorem*{thm*}{Theorem}

\theoremstyle{definition}
\newtheorem{defi}[lem]{Definition}
\theoremstyle{remark}
\newtheorem{rem}[lem]{Remark}

\allowdisplaybreaks[3]

\newcommand{\N}{\mathbb{N}}
\newcommand{\Z}{\mathbb{Z}}
\newcommand{\Q}{\mathbb{Q}}
\newcommand{\R}{\mathbb{R}}
\newcommand{\C}{\mathbb{C}}

\renewcommand{\lvert}{\left\vert}
\renewcommand{\rvert}{\right\vert}
\renewcommand{\lVert}{\left\Vert}
\renewcommand{\rVert}{\right\Vert}

\title[Sum of digits function of linear recurrence number systems]{The level of distribution of the sum-of-digits function of linear recurrence number systems}

\author[M. G. Madritsch]{Manfred G. Madritsch}
\address[M. G. Madritsch]{
\noindent 1. Universit\'e de Lorraine, Institut Elie Cartan de Lorraine, UMR 7502, Vand\-oeuvre-l\`es-Nancy, F-54506, France;\newline
\noindent 2. CNRS, Institut Elie Cartan de Lorraine, UMR 7502, Vandoeuvre-l\`es-Nancy, F-54506, France}
\email{manfred.madritsch@univ-lorraine.fr}

\author[J. M. Thuswaldner]{J\"org M. Thuswaldner}
\address[J. M. Thuswaldner]{Department of Mathematics and Information Technology, University of Leoben, Franz-Josef-Strasse 18, A-8700 Leoben, Austria}
\email{joerg.thuswaldner@unileoben.ac.at}

\dedicatory{To the memory of Christian Mauduit}

\thanks{The first author was supported by project ANR-18-CE40-0018 funded by
the French National Research Agency. The second author was supported by projects P27050 and P29910 funded by the Austrian Science Fund.}



\date{\today}

\begin{document}

\begin{abstract}
Let $G=(G_j)_{j\ge 0}$ be a strictly increasing linear recurrent sequence of integers with $G_0=1$ having characteristic polynomial $X^{d}-a_1X^{d-1}-\cdots-a_{d-1}X-a_d$. It is well known that each positive integer $\nu$ can be uniquely represented by the so-called \emph{greedy expansion} $\nu=\varepsilon_0(\nu)G_0+\cdots+\varepsilon_\ell(\nu)G_\ell$ for $\ell \in \N$ satisfying $G_\ell \le \nu < G_{\ell+1}$. 
Here the \emph{digits} are defined recursively in a way that $0\le \nu - \varepsilon_{\ell}(\nu) G_\ell - \cdots - \varepsilon_j(\nu) G_j < G_j$ holds for $0  \le j \le \ell$. In the present paper we study the \emph{sum-of-digits function} $s_G(\nu)=\varepsilon_0(\nu)+\cdots+\varepsilon_\ell(\nu)$ under certain natural assumptions on the sequence $G$. In particular, we determine its \emph{level of distribution} $x^{\vartheta}$. To be more precise, we show that for $r,s\in\N$ with $\gcd(a_1+\cdots+a_d-1,s)=1$ we have for each $x\ge 1$ and all $A,\varepsilon\in\R_{>0}$ that
\[
\sum_{q<x^{\vartheta-\varepsilon}}\max_{z<x}\max_{1\leq h\leq q}
  \lvert\sum_{\substack{k<z,s_G(k)\equiv r\bmod s\\ k\equiv h\bmod q}}1
  -\frac1q\sum_{k<z,s_G(k)\equiv r\bmod s}1\rvert
  \ll x(\log 2x)^{-A}.
\]
Here $\vartheta=\vartheta(G) \ge \frac12$ can be computed explicitly and we have $\vartheta(G) \to 1$ for $a_1\to\infty$. As an application we show that 
$\#\{
k\le x \,:\, s_G(k) \equiv r \pmod{s}, \; k \hbox{ has at most two prime factors} \} \gg x/\log x
$ provided that the coefficient $a_1$ is not too small. Moreover, using Bombieri's sieve an ``almost prime number theorem'' for $s_G$ follows from our result. 

Our work extends earlier results on the classical $q$-ary sum-of-digits function obtained by Fouvry and Mauduit.
\end{abstract}

\maketitle

\section{Introduction}
The present paper is devoted to arithmetic properties of the
sum-of-digits function $s_G$ of a digit expansion with respect to a
sequence $G=(G_j)_{j\geq 0}$ which is defined in terms of a linear recurrence relation. We
establish a version of the theorem of Bombieri and Vinogradov for $s_G$ (for the classical version of this theorem we refer {\it
  e.g.}~to~\cite[Theorem~9.18]{FI:10}). In other words, we provide a result on the \emph{level of distribution} of $s_G$ (see for instance~\cite[Chapters~3,~22, and~25]{FI:10} or Greaves~\cite[Chapter~5]{Greaves:01} for information on this notion).  From this result we derive distribution results for $s_G$ on the set of integers having few prime factors. Our tools comprise exponential sum estimates and sieve methods. What we do here extends results of Fouvry and Mauduit~\cite{fouvry_mauduit1996:methodes_de_crible,Fouvry-Mauduit:96}, where the level of distribution of the $q$-ary sum-of-digits function is investigated (see also the recent preprint of Spiegelhofer~\cite{S:19}). Our results provide a first step towards a generalization of the beautiful work of Mauduit and Rivat~\cite{MR:10} on the $q$-ary sum of digits of primes to digit expansions w.r.t.\ a linear recurrent sequence. We mention that new ideas are needed in our setting in order to establish the exponential sum estimates necessary for proving our main results.

\subsection{Linear recurrence number systems} We start with the
definition of digit expansions w.r.t.\ a sequence of integers. Let
$G=(G_j)_{j\geq 0}$ be a strictly increasing sequence of positive
integers and suppose that $G_0=1$. Using the \emph{greedy algorithm}
one can associate a unique digit expansion to each positive integer $\nu$
w.r.t.\ this sequence $G$. Indeed, for each integer
$\nu \ge 1$ there exists a unique $\ell\in\mathbb{N}$ such that
$G_\ell\leq \nu< G_{\ell+1}$.  With this number $\ell$ we can define the
\emph{digits} $\varepsilon_{\ell}(\nu),\ldots,\varepsilon_0(\nu)$
recursively in a way that
\[
0\le \nu - \varepsilon_{\ell}(\nu) G_\ell - \cdots - \varepsilon_j(\nu) G_j < G_j \qquad (0 \le j \le \ell).
\]
This leads to the digit expansion
\begin{equation}\label{eq:greedyexp}
\nu=\varepsilon_0(\nu)G_0+\cdots+\varepsilon_\ell(\nu)G_\ell
\end{equation}
for $\nu$ w.r.t.\ the sequence $G$. It is easy to check that we have
$0\leq\varepsilon_j(\nu)<\frac{G_{j+1}}{G_j}$ for each $0\le j\le \ell$
and that this expansion is unique with the property that
\[
0 \le \varepsilon_0(\nu)G_0+\cdots+\varepsilon_j(\nu)G_j<G_{j+1}
\]
for $0\leq j\leq \ell$. Using the greedy expansion for the
sequence $G$, we define the sum-of-digits function of $\nu$ w.r.t.\ $G$
by
\begin{equation}\label{eq:sg}
s_G(\nu)=\varepsilon_0(\nu)+\cdots+\varepsilon_\ell(\nu)\qquad(\nu\ge 1)
\end{equation}
and by setting $s_G(0)=0$ for convenience.
In the present paper we deal with sequences $G$ that are defined
in terms of linear recurrences. This idea goes back to
Zeckendorf~\cite{Zeckendorf:72} for the case of Fibonacci numbers (see
{\em
  e.g.}~\cite{Drmota-Gajdosik:98,Lamberger-Thuswaldner:03,Petho-Tichy:89}
for the general case). We recall the following definition.

\begin{defi}[Linear recurrence base]\label{def:lrb}
  We will refer to a strictly increasing sequence $G=(G_j)_{j\geq0}$
  as a {\em linear recurrence base}, if there exist
  $a_1,\ldots,a_d\in \mathbb{N}$ with $a_d>0$ such that the following
  conditions hold:
  \begin{enumerate}
  \item \label{it:it1} $G_0=1$ and $a_1G_{k-1}+\cdots+a_kG_0<G_{k}$
    for $1\leq k< d$.
  \item \label{it:it2} $G_{n+d}=a_1G_{n+d-1}+\cdots+a_dG_n$ holds for
    each $n\in \mathbb{N}$.
  \item \label{it:it3}
    $(a_k,a_{k+1},\ldots,a_d)\preceq(a_1,a_2,\ldots,a_{d-k+1})$ for
    $1<k\leq d$, where ``$\prec$'' indicates the lexicographic
    order.
  \end{enumerate}
  The polynomial $X^d-a_1X^{d-1}-\cdots- a_1X-a_0$ is called the \emph{characteristic polynomial} of the linear recurrence base $G$. Its dominant root (which is a positive real number) is called $\alpha$.
\end{defi}

We want to make some comments on this definition which is the same as the one used in Lamberger and
Thuswaldner~\cite{Lamberger-Thuswaldner:03}. Item~\eqref{it:it3} immediately yields that $a_1\ge \max\{a_2,\ldots, a_d\}$.
Moreover, our conditions imply with the same proof as Steiner~\cite[Lemma~2.1]{Steiner:00} that
\[
G_{n+d-k} > a_{k+1}G_{n+d-k-1} + \cdots + a_dG_n \qquad(n\in\N, \, 1\le k \le d-1),
\]
a condition that was used for instance in Drmota and Gajdosik~\cite{Drmota-Gajdosik:98a,Drmota-Gajdosik:98}. In \cite[Lemma~3.1]{Drmota-Gajdosik:98a} it is proved (under milder conditions than ours) that the characteristic polynomial of $G$ has a dominant root $\alpha>1$ and, because all coefficients of the recurrence satisfy $0\le a_j \le a_1$ in our case, we even have
\begin{equation}\label{eq:alpharange}
\alpha \in [a_1,a_1+1). 
\end{equation}
The fact that $\alpha$ is dominant yields that there are constants $c,\delta\in\R_{>0}$ such that 
\begin{equation}\label{eq:recasympt}
G_n = c\alpha^n + \mathcal{O}(\alpha^{(1-\delta)n}) \qquad (n\ge 0).
\end{equation}

If item~\eqref{it:it1} is strengthened to $G_0=1$ and $a_1G_{k-1}+\cdots+a_kG_0+1=G_{k}$ for $1\leq k< d$, according to \cite[Proposition~2.1]{Steiner:00} the string $\varepsilon_0,\ldots,\varepsilon_\ell$ can occur as a digit string in
\eqref{eq:greedyexp} if and only if
$
(\varepsilon_{j},\ldots,\varepsilon_{j+d-1})\prec(a_1,a_2,\ldots,a_{d})
$
holds for $0\le j\le \ell$ (here we have to pad $\varepsilon_0,\ldots,\varepsilon_\ell$ with
$d-1$ zeros).  This is called the {\em Parry-condition} and goes back
to Parry~\cite{parry1960:eta_expansions_real} where it was introduced
in the context of beta-numeration. We also mention that in some earlier papers on
linear recurrence bases instead of item~\eqref{it:it3} the stronger
condition $a_1\ge a_2 \ge \cdots \ge a_d >0$ is assumed (see {\it
  e.g.}~\cite{Grabner-Tichy:90, Petho-Tichy:89}). 

A linear recurrence base together with the associated digit expansions \eqref{eq:greedyexp} will be called a \emph{linear recurrence number system}.

\subsection{Previous results}
The most prominent example of a linear recurrence base is the {\it Fibonacci sequence} $F=(F_j)_{j\ge 2}$ defined by $F_0=0$, $F_1=1$, and $F_{n+2}=F_{n+1} + F_{n}$ for $n\ge 0$ (note that we have to start with index $j=2$ in the sequence $F$ to meet the conditions of Definition~\ref{def:lrb}). The associated linear recurrence number system was first studied by Zeckendorf~\cite{Zeckendorf:72}. For this reason expansions of the shape \eqref{eq:greedyexp} are called {\it Zeckendorf expansions} in this case. In the meantime linear recurrence number systems received a lot of attention and have been studied by many authors. Without making an attempt to be complete we mention a few results on linear recurrence number systems with special emphasis on the sum-of-digits function $s_G$ defined in~\eqref{eq:sg}.

Peth\H{o} and Tichy~\cite{Petho-Tichy:89} provide an asymptotic formula of the summatory function of $s_G$. Using analytic methods and results from Coquet, Rhin, and Toffin~\cite{Coquet-Rhin-Toffin:81}, Grabner and Tichy~\cite{Grabner-Tichy:90} prove that $(z s_G(n))_{n\in\N}$ is equidistributed modulo $1$ for each $z \in\R\setminus\Q$. By elementary exponential sum estimates Lamberger and Thuswaldner~\cite{Lamberger-Thuswaldner:03} establish distribution results of $s_G(n)$ in residue classes and derive some consequences including a Barban-Davenport-Halberstam type theorem for $s_G$.  Distribution functions for so-called {\it $G$-additive functions} (a natural generalization of $s_G$ analogous to the well-known $q$-additive functions) are investigated by Barat and Grabner~\cite{BG:96}. In~\cite{BG:96} the authors also provide a dynamic approach to linear recurrence number systems on the {\it $G$-compactification} $\mathcal{K}_G$ on which a dynamical system can be defined in terms of the addition of $1$; this {\it $G$-odometer} goes back to Grabner {\it et al.}~\cite{GLT:95} (see also \cite{BG:16} for a more recent study of this object).  A local limit law for $s_G$ is proved by Drmota and Gajdosik~\cite{Drmota-Gajdosik:98a}. In \cite{Drmota-Gajdosik:98} the same authors consider sums of the shape $\sum_{\nu<N} (-1)^{s_G(\nu)}$. Drmota and Steiner~\cite{DS:02,Steiner:02} establish a central limit theorem for $G$-additive functions along polynomial sequences, and Wagner~\cite{Wagner:07} studies properties of sets of numbers $\nu<N$ characterized by the fact that $s_G(\nu)=k$ for some fixed positive integer $k$. Recently, Miller and his co-authors proved further distribution results related to linear recurrence number systems. See for example \cite{BBGILMT}, where run lengths of zeros in Zeckendorf expansions are studied, or \cite{BDEMMKTW}, which is concerned with the number of nonzero digits in Zeckendorf expansions. Motivated by the proof of Gelfond's old conjecture on the distribution of the sum-of-digits function of primes in residue classes by Mauduit and Rivat~\cite{MR:10} and, more generally, by Sarnak's conjecture \cite{Sarnak:12}, the question whether $s_G$ has nice distribution properties for prime arguments came into the focus of research. We mention that M\"obuis orthogonality of $s_F$ is proved in the Zeckendorf case by Drmota {\it et al.}~\cite{DMS:18}. The exponential sum methods developed in \cite{MR:10} also led to a wealth of new results on sum-of-digits functions. In the context of Zeckendorf expansions the joint distribution of the ordinary $q$-ary sum-of-digits function and $s_F$ is investigated by Spiegelhofer~\cite{S:14} by using methods in the spirit of~\cite{MR:10}.  Finally, we note that, starting with Barat and Grabner~\cite{BG:96}, van der Corput and Halton type sequences using linear recurrence bases are investigated. Work on this topic can be found in Ninomiya~\cite{Ninomiya:98},  Hofer {\it et al.}~\cite{HIT}, and Thuswaldner~\cite{Thuswaldner:17}.

We mention that Ostrowski expansions~\cite{Berthe:01} as well as beta-expansions~\cite{FS:92,parry1960:eta_expansions_real,Renyi:57} are related to linear recurrence number systems. 

\subsection{Statement of results and associated exponential sums}\label{sec:13}

Let $G=(G_j)_{j\geq0}$ be a linear recurrence base satisfying the conditions of Definition~\ref{def:lrb}. The aim of the present article is to study the level of distribution $x^{\vartheta(G)}$ of the sum-of-digits function $s_G$. In other words, our main result is the extension of \cite[Th\'eor\`eme]{fouvry_mauduit1996:methodes_de_crible} to linear recurrence bases. 

\begin{thm}\label{thm:bombieri-vinogradov-type}
Let $G=(G_j)_{j\geq0}$ be a linear recurrence base with characteristic polynomial $X^{d}-a_1X^{d-1}-\cdots-a_{d-1}X-a_d$ 
%with dominant root $\alpha>0$ 
satisfying the conditions of Definition~\ref{def:lrb}. Let $r,s \in \N$ with $\gcd(a_1+\cdots+a_d-1,s)=1$. Then for each $x\ge 1$and all $A,\varepsilon\in\R_{>0}$, we have 
\begin{equation}\label{eq:thm11}
\sum_{q<x^{\vartheta-\varepsilon}}\max_{z<x}\max_{1\leq h\leq q}
  \lvert\sum_{\substack{k<z\\s_G(k)\equiv r\bmod s\\ k\equiv h\bmod q}}1
  -\frac1q\sum_{\substack{k<z\\s_G(k)\equiv r\bmod s}}1\rvert
  \ll x(\log 2x)^{-A},
\end{equation}
where the implied constant depends on $\varepsilon$ and $A$. 
Here $\vartheta=\vartheta(G) \ge \frac12$ can be computed explicitly and we have $\vartheta(G) \to 1$ for $a_1\to\infty$.
\end{thm}

\begin{rem}
We are able to give concrete values for $\vartheta(G)$. Let $\alpha$ be the dominant root of the characteristic polynomial of $G$. We show  that $\vartheta(G) \ge \max\{\frac12, 1- \log_\alpha(m_G+3)\}$ for $m_G$ as in \eqref{eq:ma0}.  Since $m_G \ll \log a_1 \ll \log \alpha$ by  Lemma~\ref{lem:mest} this already implies that  $\vartheta(G) \to1$ for $a_1\to\infty$. On top of this, in Lemma~\ref{lem:ThetaImproved} we give better estimates for $\vartheta(G)$ for small values of $a_1$. These estimates are needed in order to prove Corollary~\ref{cor1} below.
\end{rem}

Similarly as Fouvry and Mauduit~\cite{fouvry_mauduit1996:methodes_de_crible} we can deduce two applications of Theorem~\ref{thm:bombieri-vinogradov-type}. The first one deals with the distribution of the sum-of-digits function $s_G$ evaluated along almost primes. 

\begin{cor}\label{cor1}
Let $G=(G_j)_{j\geq0}$ be a linear recurrence base with characteristic polynomial $X^{d}-a_1X^{d-1}-\cdots-a_{d-1}X-a_d$ satisfying the conditions of Definition~\ref{def:lrb}. Let $r,s \in \N$ with $\gcd(a_1+\cdots+a_d-1,s)=1$. Then for $a_1 \ge 59$ we have 
\begin{equation}\label{eq:cor1}
\#
\{
k\le x \; :\; s_G(k) \equiv r \pmod{s}, \; k= p_1 \hbox{ or } k=p_1p_2 \hbox{ with } p_1,p_2 \hbox{ prime}  
\} \gg \frac{x}{\log x}
\end{equation}
for $x\to\infty$. If the characteristic polynomial of $G$ has the particular form $X^2-a_1X-1$ then this result even holds for $a_1\ge 15$.
\end{cor}

It is well known (see for instance Greaves~\cite[Chapter~5]{Greaves:01}) that results on the level of distribution of a set $A(x)$ of positive integers less than $x$ can be used to get results on the number of almost primes contained in $A(x)$. In particular, if the level of distribution of $A(x)$ is $x^{\vartheta-\varepsilon}$ with $\vartheta$ large enough to satisfy $\frac {1}{\vartheta} < 2 - \delta_2$  for a certain constant $\delta_2$, then the number of almost primes in $A(x)$ can be estimated from below by a constant times $\frac{x}{\log x}$. There has been a lot of effort to get the constant $\delta_2$ as small as possible. To our knowledge, currently the best value is $\delta_2=0.044560$ and this is due to  Greaves~\cite{Greaves:86} (although $\delta_2$ is conjectured to be equal to $0$). Thus in order to prove Corollary~\ref{cor1} we need to make sure that 
\begin{equation}\label{eq:thetabound}
\vartheta(G) > 0.5113938... = 1-0.4886061... 
\end{equation}
for the linear recurrence bases indicated in its statement. The lower bound $59$ (resp.\ $15$) for $a_1$ is an artifact of the methods we are using in the proof. However, in principle our method allows (with sufficient computation power) to extend the result to smaller values of $a_1$ (see Section~\ref{sec:smaller-traces} for details on this). However, we do not think that it is feasible to get the result for $a_1=1$ with present time computers.

Our second corollary provides a prime number theorem for numbers whose sum-of-digits function $s_G$ lies in a prescribed residue class. Analogously to the case of the ordinary $q$-ary sum-of-digits function (see \cite[Corollaire~2]{fouvry_mauduit1996:methodes_de_crible}) this corollary gives a nontrivial result only for large values of $a_1$. In the following statement $\Lambda_\ell = \mu * \log^\ell$ denotes the generalized von Mangoldt function ($\ell \ge 1$; here $\mu$ is the M\"obius function and ``$*$'' denotes Dirichlet convolution).

\begin{cor}\label{cor2}
Let $G=(G_j)_{j\geq0}$ be a linear recurrence base with characteristic polynomial $X^{d}-a_1X^{d-1}-\cdots-a_{d-1}X-a_d$ satisfying the conditions of Definition~\ref{def:lrb}. Let $\ell,r,s \in \N$ with $\ell \ge 2$ and  $\gcd(a_1+\cdots+a_d-1,s)=1$. Then there is $x_0=x_0(G,s,\ell)$ such that for $x\ge x_0$ we have
\[
\sum_{\substack{k<x \\ s_G(k)\equiv r\bmod s}} \Lambda_\ell(k) = \frac{\ell}{s}x(\log x)^{\ell-1}
\Big(
1+ \mathcal{O} \Big(\frac{(\log\log a_1)^5}{\log a_1}\Big)
\Big),
\]
where the implied constant depends only on $s$ and $\ell$.
\end{cor} 

Corollary~\ref{cor2} follows from Theorem~\ref{thm:bombieri-vinogradov-type} by an application of the sieve of Bombieri ({\it cf}.~\cite[Theorem~3.5]{FI:10}). Since the proof of Corollary~\ref{cor2} is {\it verbatim} the same as the one of \cite[Corollaire~2]{fouvry_mauduit1996:methodes_de_crible} in \cite[Section~VII]{fouvry_mauduit1996:methodes_de_crible} we do not reproduce it here.

\smallskip

The paper is organized as follows. In Section~\ref{sec:rewriting-problem} we reduce the problem of proving Theorem~\ref{thm:bombieri-vinogradov-type} to an exponential sum estimate and provide some preliminaries. Section~\ref{sec:estimates} is devoted to the estimate of the exponential sums needed in the proof. In Section~\ref{sec:smaller-traces} we give a computer assisted improvement for these estimates to make them applicable for small values of the coefficient $a_1$. Using these preparations in Section~\ref{sec:proof} we provide the proof of  Theorem~\ref{thm:bombieri-vinogradov-type} and of Corollary~\ref{cor1}. Moreover, we provide an estimate for $\vartheta(G)$ for small values of $a_1$.

\section{Rewriting the problem}\label{sec:rewriting-problem}
Let $G=(G_j)_{j\geq0}$ be a linear recurrence base with characteristic polynomial $X^{d}-a_1X^{d-1}-\cdots-a_{d-1}X-a_d$ satisfying the conditions of Definition~\ref{def:lrb}. The proof of Theorem~\ref{thm:bombieri-vinogradov-type} relies on exponential sums. Setting $e(z)=\exp(2\pi \sqrt{-1}  z)$ we get for integers $a,b,c$ with $c\geq1$ that
\[
\frac1c\sum_{h=1}^ce\left(\frac
    hc(a-b)\right)=\begin{cases}1&\text{if }a\equiv b\bmod c,\\0&\text{otherwise}.\end{cases}
\]
Thus the difference inside the absolute value of \eqref{eq:thm11} may be written as
\begin{align*}
  R(z)=R(z;u,q,r,s)&=
  \sum_{\substack{k<z\\s_G(k)\equiv r\bmod s\\ k\equiv u\bmod
      q}}1
  -\frac1q\sum_{\substack{k<z\\s_G(k)\equiv r\bmod s}}1\\
  &=\frac1{sq}\sum_{b=1}^s\sum_{h=1}^{q-1}\sum_{k<z}
   e\left(\frac{b}{s}(s_G(k)-r)+\frac {h}{q}(k-u)\right).
\end{align*}
Splitting the contribution of $b=s$ apart we get that
\begin{equation}\label{eq:R2}
R(z)=\frac1{sq}\sum_{b=1}^{s-1}\sum_{h=1}^{q-1}
  e\left(-\frac{br}s-\frac{uh}q\right) \sum_{k<z}e\left(\frac{b}{s}s_G(k)+\frac{h}{q}k\right)+\mathcal{O}\left(\frac{q}{s}\right).
\end{equation}

In view of \eqref{eq:R2} the proof of Theorem \ref{thm:bombieri-vinogradov-type} boils down
to showing that
\begin{gather*}%\label{rewritten-estimatePRAE}
  \sum_{Q<q\leq 2Q} \sum_{h=1}^{q-1} \left|    \sum_{k<z}e\left(\frac{r}{s}s_G(k)+\frac{h}{q}k\right)   \right| \ll Qx(\log
  2x)^{-A}
\end{gather*}
holds for each $A>0$ if $1\le r\le s-1$, $Q\leq x^{\vartheta(G)-\varepsilon}$, and $z<x$. To make our proofs easier we want to subdivide the sum over $k$ according to the greedy expansion \eqref{eq:greedyexp} of $z$ (of course we may assume w.l.o.g.\ that $z$ is a positive integer). Since $z < x$ there is $N \le \log_\alpha x+C$ (for some constant $C$ depending on $G$) such that
\begin{equation}\label{eq:greedyy}
z = \sum_{0\le n\le N} \varepsilon_n(z) G_n.
\end{equation}
For $y,\beta\in[0,1]$ we define the following exponential sum
\[
S_{n}(y,\beta):=\sum_{k<G_n}e\left(\beta s_G(k)+y k\right).
\]
Using  \eqref{eq:greedyy} we gain by splitting off one digit of $z$ after the other (like it is done for instance in the proof of \cite[Lemma~1]{Grabner-Tichy:90}),
\begin{align*}
\left|\sum_{k<z}e\left(\frac{r}{s}s_G(k)+\frac{h}{q}k\right)\right|
\le  \sum_{n=0}^N \left|  \varepsilon_n(y) S_n\left(\frac {h}{q},\frac{r}{s}\right) \right| \ll  \sum_{n=0}^N \left|S_n\left(\frac {h}{q},\frac{r}{s}\right)\right|.
\end{align*}
Thus, since $N\le \log_\alpha x+C$, Theorem \ref{thm:bombieri-vinogradov-type} follows  if we prove 
\begin{gather}\label{rewritten-estimate}
  \sum_{Q<q\leq 2Q} \sum_{h=1}^{q-1} \left|S_n\left(\frac {h}{q},\frac{r}{s}\right)\right| \ll Qx(\log
  2x)^{-A}
\end{gather}
for each $A>0$  if $1\le r\le s-1$, $Q\leq x^{\vartheta(G)-\varepsilon}$, and $n\le \log_\alpha x+C$.

We start by setting up a recurrence relation for $S_{n}(y,\beta)$. 
Let
\begin{equation}\label{eq:I}
\mathcal{I}:=\{1\leq j\leq d\colon a_j\neq0\}
\end{equation}
be the set of indices corresponding to non-vanishing coefficients of the characteristic polynomial of $G$. As $a_d>0$, item~\eqref{it:it3} of Definition~\ref{def:lrb} implies that
$\{1,d\}\subset\mathcal{I}$. Then the exponential sum $S_{n}(y,\beta)$
satisfies the recurrence
\begin{equation}\label{eq:Srec}
S_n(y,\beta)=\sum_{j\in\mathcal{I}}A_{n,j}(y,\beta)S_{n-j}(y,\beta)
\end{equation}
with
\begin{equation}\label{eq:A}
A_{n,j}(y,\beta)=\sum_{\ell=0}^{a_j-1}e\bigg(y\bigg(\sum_{k=1}^{j-1}a_kG_{n-k}+\ell G_{n-j}\bigg)+\beta\bigg(\sum_{k=1}^{j-1}a_k+\ell \bigg)\bigg)
\end{equation}
for $1\le j\le d$ (see \cite[Equation
(3)]{Lamberger-Thuswaldner:03}). Iterating this recurrence relation we obtain
\begin{equation*}\label{eq:iterate}
  \begin{split}
    S_{n}(y,\beta)&=\sum_{j\in\mathcal{I}}A_{n,j}(y,\beta)S_{n-j}(y,\beta)\nonumber\\
    &=\sum_{j_1,j_2\in\mathcal{I}}A_{n,j_1}(y,\beta)A_{n-j_1,j_2}(y,\beta)S_{n-j_1-j_2}(y,\beta)\\
    &=\sum_{j_1,\ldots,j_k\in\mathcal{I}}A_{n,j_1}(y,\beta)\cdots
    A_{n-j_1-\cdots-j_{k-1},j_k}(y,\beta)S_{n-j_1-\cdots-j_k}(y,\beta),
  \end{split}
\end{equation*}
which makes sense as long as $n-j_1-\cdots-j_{k-1} \geq d$ holds for
all constellations $(j_1,\ldots,j_{k-1})\in\mathcal{I}^{k-1}$. For
$d\leq n_0< n$ and $1\leq k < n$ let 
%$J_k(n_0)$ be the set 
%of iterations $(j_1,\ldots,j_k)$ that step over $n_0$ with the last entry $j_k$, \textit{i.e.},
\begin{equation}\label{eq:iterate2}
J_k(n_0)=\left\{\mathbf{j}=(j_1,\ldots,j_k)\in\mathcal{I}^k\colon
  n-\sum_{\ell=1}^{k-1} j_\ell>n_0\geq n-\sum_{\ell=1}^{k}j_\ell\right\}.
\end{equation}
Then
\begin{equation}\label{eq:Iprodbis}
    \lvert S_{n}(y,\beta)\rvert
    \leq\sum_{k=1}^{n-n_0}\sum_{(j_1,\ldots,j_k)\in
      J_k(n_0)}\prod_{\ell=1}^k\lvert
    A_{n-\sum_{r=1}^{\ell-1}j_r,j_\ell}(y,\beta)\rvert \cdot
    \lvert S_{n-\sum_{r=1}^{k}j_r}(y,\beta)\rvert.
\end{equation}

The central idea in proving \eqref{rewritten-estimate} is a combination of max- and $1$-norm estimates of $S_n(y,\beta)$ and related expressions. 

\section{Estimates of exponential sums related to $S_n(y,\beta)$}\label{sec:estimates}
We subdivide this section into three parts. First we consider the $1$-norm of $S_{n}(\cdot,\beta)$ and of its derivative.  These $1$-norms play a role in the proof of Theorem~\ref{thm:bombieri-vinogradov-type} after an application of an inequality due to Sobolev and Gallagher which is an important tool in the context of the large sieve (see Lemma~\ref{mo:lem1.2} below for its statement). In the second part we estimate the maximum-norm of sums of certain products related to $S_{n}(y,\beta)$. 
Finally the third part deals with an estimation of a parameter which occurs in our estimate of the $1$-norm of $S_{n}(\cdot,\beta)$.

\subsection{The $1$-norm of $S_{n}(\cdot,\beta)$}\label{sec:1n}
Let $G=(G_j)_{j\ge 0}$ be a linear recurrence base with characteristic polynomial $X^{d}-a_1X^{d-1}-\cdots-a_{d-1}X-a_d$
satisfying the conditions of Definition~\ref{def:lrb}. We set $a=a_1$ and let $\beta\in\R$ be fixed.
Define for $k\in \N$, $j\in \mathcal{I}$ with $k\ge j$, and $y\in\R$ the functions
\[
f_{k,j}(y)=
\begin{cases}
  \lvert\frac{\sin\left(\pi a_j(\beta+yG_{k-j})\right)}{\sin\left(\pi(\beta+yG_{k-j})\right)}\rvert &  \hbox{if } \beta+yG_{k-j} \not\in \mathbb{Z}, \\
  a_j & \hbox{if } \beta+yG_{k-j} \in \mathbb{Z}.
\end{cases}
\]
This permits us to write the modulus of the sums $A_{k,j}(y,\beta)$ in \eqref{eq:A}  as
\begin{equation}\label{eq:Af}
\lvert A_{k,j}(y,\beta)\rvert
=\lvert\sum_{\ell=0}^{a_j-1}e\left(\ell(\beta +yG_{k-j})\right)\rvert
=f_{k,j}(y).
\end{equation}
We note that the numerator of $f_{k,j}(y)$ has period
$(a_jG_{k-j})^{-1}$. 

For each $k\in \N$ we subdivide the interval
$\left[-\frac\beta{G_{k}},1-\frac{\beta}{G_{k}}\right)$ (which is the same
as $[0,1)$ modulo $1$) into $aG_{k}$ parts
\begin{equation}\label{eq:partition}
I_{k}(b)=\left[\frac{b-a\beta}{aG_{k}},\frac{b+1-a\beta}{aG_{k}}\right) \qquad(0\leq b<aG_{k})
\end{equation}
of equal length $(aG_{k})^{-1}$. In each of the intervals $I_{k-j}(b)$
the supremum of $f_{k,j}(y)$ satisfies
\begin{equation}\label{eq:supn}
\sup_{y\in I_{k-j}(b)}f_{k,j}(y) = m(j,b) \qquad(0\leq b<a G_{k-j}),
\end{equation}
with
\begin{equation}\label{eq:mab}
m(j,b)=m_G(j,b)=\sup_{y\in (\frac b{a},\frac{b+1}{a})}\left\vert\frac{\sin \pi a_j y}{\sin \pi y}\right\vert \qquad(j\in\mathcal{I},\, b\in\Z).
\end{equation}
Thus the supremum of $\lvert A_{k,j}(y,\beta)\rvert=f_{k,j}(y)$ in (\ref{eq:supn}) is independent of $k$. 
It is immediate that for
$b\equiv 0,a-1\pmod{a}$ this supremum is equal to $a_j$ (it is attained
for $b\equiv 0 \pmod{a}$ on the left endpoint of $I_{k-j}(b)$, and for
$b\equiv a-1\pmod{a}$ for the limit towards the right endpoint of
$I_{k-j}(b)$). If $j=1$ and $b\not\equiv 0,a-1\pmod{a}$ then
$f_{k,1}(y)$ is a unimodal function on $I_{k-1}(b)$ which is equal to
zero at its endpoints and whose global maximum is the unique local
maximum in that interval. 

We define the piecewise constant function
\begin{equation}\label{eq:Fdef}
F_{k,j}(y)= m(j,b) \quad\hbox{for}\quad y \in  I_{k-j}(b) \qquad (0 \le b < aG_{k-j}),
\end{equation}
which forms an upper bound for $f_{k,j}(y)$. The functions
$f_{k,1}(y)$ and $F_{k,1}(y)$ are plotted in Figure~\ref{fig:Psi} for a special set of
parameters.

We will also need the following notations. With $m(j,b)$ as in \eqref{eq:mab} we define
\begin{equation}\label{eq:ma}
m(j)=m_G(j)= \frac1a\sum_{b=0}^{a-1} m_G(j,b) \qquad(j\in\mathcal{I}).
\end{equation}
and finally
\begin{equation}\label{eq:ma0}
m=m_G=\max_{j\in \mathcal{I}} m_G(j).
 \end{equation}

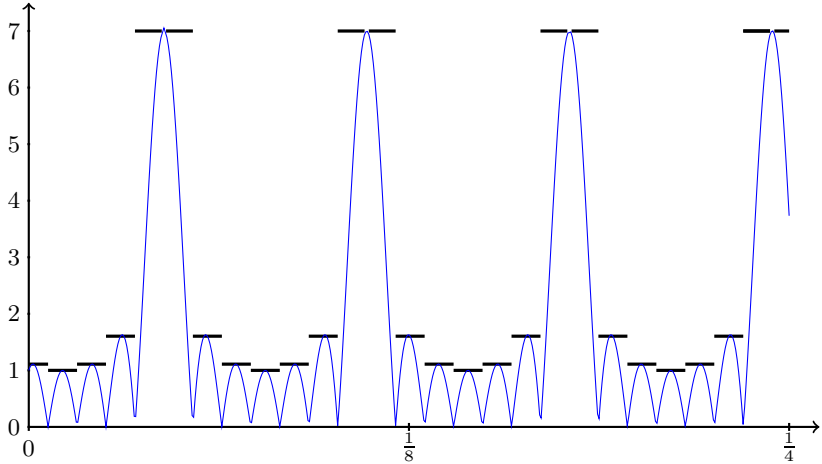
\begin{figure}[ht]
\begin{center}
\begin{tikzpicture}
\begin{scope}[xscale=40,yscale=0.75]
\def\ya{7.0};
\def\yb{{abs(1/sin(3*pi/14 r))}};
\def\yc{{abs(1/cos(pi/7 r))}};
\def\yd{1.0};

\draw[very thick] (0,\yc)--(2/315,\yc) (74/315,\ya)--(1/4,\ya);

\foreach \i in {11,14,32,35,53,56,74}
  \draw[very thick] (\i/315,\ya)--({(\i+3)/315},\ya);
\foreach \i in {8,17,29,38,50,59,71}
  \draw[very thick] (\i/315,\yb)--({(\i+3)/315},\yb);
\foreach \i in {5,20,26,41,47,62,68}
  \draw[very thick] (\i/315,\yc)--({(\i+3)/315},\yc);
\foreach \i in {2,23,44,65}
  \draw[very thick] (\i/315,\yd)--({(\i+3)/315},\yd);
\foreach \i in {14,35,56,77}
  \draw[ultra thick,white] (\i/315,0.5)--(\i/315,7.5);

\draw[thick,->] (0,0)--(0.26,0);
\draw[thick] (0,0)--(0,-2pt) node [below] {\small $0$};
\draw[thick] (1/8,-2pt)--(1/8,2pt) node [below] {\small $\tfrac18$};
\draw[thick] (1/4,-2pt)--(1/4,2pt) node [below] {\small $\tfrac14$};
\draw[thick,->] (0,0)--(0,7.5);
\foreach \y in {0,...,7}
  \draw[thick] (-3/6741,\y)--(3/6741,\y) node [left] {\small $\y$};

\draw[blue,domain=0:1/4,samples=500] plot (\x,{abs(sin(pi*7*(1/3+\x*15) r)/sin(pi*(1/3+\x*15) r))} );
\end{scope}
\end{tikzpicture}
\end{center}
\caption{For the linear recurrence base $(G_j)_{j\ge 0}$
  defined by $G_0=1$, $G_1=8$, and $G_{n+2}=7G_{n+1}+G_{n}$ for $n\ge 0$ this image shows the function
  $f_{3,1}(y)$ together with its piecewise constant upper bound
  $F_{3,1}(y)$ in the interval $y\in[0,\frac14]$ (here we chose
  $\beta =\frac13$).  \label{fig:Psi}}
\end{figure}

It will turn out that the $1$-norm of $S_{n}(\cdot ,\beta)$ can be estimated in
terms of an integral over products of the functions $F_{k,j}(y)$ 
%({\it cf.}\ \eqref{eq:Iprodbis} and the definition of $F_{k,j}(y)$ in \eqref{eq:Fdef}). 
Thus we deal with such products in our first proposition.

\begin{prop}\label{prop:Fprod} 
  Let $G=(G_j)_{j\ge 0}$ be a linear recurrence base with
  characteristic polynomial $X^{d}-a_1X^{d-1}-\cdots-a_{d-1}X-a_d$
  satisfying the conditions of Definition~\ref{def:lrb}. Fix
  $k\in \mathbb{N}$ and let $n_0,n_1,\ldots,n_k$ be a strictly
  increasing sequence of integers satisfying
  $j_{\ell}:=n_\ell-n_{\ell-1}\in\mathcal{I}$ for $1\le \ell\le
  k$. Then
\begin{equation}\label{eq:pFprod}
  \int_0^1 \prod_{\ell=1}^k f_{n_\ell, j_\ell}(y) \mathrm{d}y\le \int_0^1 \prod_{\ell=1}^k F_{n_\ell, j_\ell}(y) \mathrm{d}y \ll (m+2)^k,
\end{equation}
where $m=m_G$ is given by \eqref{eq:ma0}.
\end{prop}

\begin{proof}
  Since the first inequality in \eqref{eq:pFprod} is obvious it
  remains to prove the second one. Throughout this proof we set
  $g_\ell = \lfloor {G_{n_\ell}}/{G_{n_{\ell-1}}} \rfloor$. From the
  definition of the intervals $I_k(b)$ it is clear that each interval
  of the form $I_{n_{\ell-1}}(b)$ can be covered by $g_\ell + 2$
  adjacent intervals of the form $I_{n_{\ell}}(b')$. To be more
  precise, there is $c\in\N$ such that
  \begin{equation}\label{eq:covering}
    I_{n_{\ell-1}}(b) \subset I_{n_\ell}(c)\cup I_{n_\ell}(c+1) \cup \dots \cup I_{n_\ell}(c+g_\ell +1).
  \end{equation}
  In the first step of our proof we subdivide
  $[-\beta/G_k, 1-\beta/G_k)$ ($\equiv [0,1)\bmod 1$) into intervals of the form
  $I_{n_0}(b)$ to obtain (recall that $a=a_1$)
  \[
    J:=\int_0^1 \prod_{\ell=1}^k F_{n_\ell, j_\ell}(y) \mathrm{d}y
    =\sum_{b_0=0}^{aG_{n_0}-1}\int_{I_{n_0}(b_0)}\prod_{\ell=1}^k
    F_{n_\ell, j_\ell}(y) \mathrm{d}y.
  \]
  Since $n_0=n_1-j_1$, by definition, we have
  $F_{n_1,j_1}(y)=m(j_1,b_0)$ for $y\in I_{n_0}(b_0)$. Thus we may
  pull this constant out of the integral yielding
  \[
    J\le
    \sum_{b_0=0}^{aG_{n_0}-1}m(j_1,b_0)\int_{I_{n_0}(b_0)}\prod_{\ell=2}^k
    F_{n_\ell, j_\ell}(y) \mathrm{d}y.
  \]
  Now we use \eqref{eq:covering} to cover each $I_{n_0}(b_0)$ by
  $g_1+2$ adjacent intervals of the form $I_{n_1}(b')$. More
  precisely, to each $b_0$ there is an integer $c_1(b_0)$ such that
  \[
    I_{n_0}(b_0) \subset I_{n_1}(c_1(b_0))\cup I_{n_1}(c_1(b_0)+1)
    \cup \ldots \cup I_{n_1}(c_1(b_0)+g_1 +1).
  \]
  Since the integrand is nonnegative this yields the estimate
  \[
    J\le \sum_{b_0=0}^{aG_{n_0}-1}m(j_1,b_0)
    \sum_{b_1=0}^{g_1+1}\int_{I_{n_1}(c_1(b_0)+b_1)}\prod_{\ell=2}^k
    F_{n_\ell, j_\ell}(y) \mathrm{d}y.
  \]
  As before we have $F_{n_2,j_2}(y)=m(j_2,c_1(b_0)+b_1)$ for
  $y\in I_{n_1}(c_1(b_0)+b_1)$ and we may pull this constant out of
  the integral again to get
  \[
    J\le \sum_{b_0=0}^{aG_{n_0}-1}m(j_1,b_0)
    \sum_{b_1=0}^{g_1+1}m(j_2,c_1(b_0)+b_1)\int_{I_{n_1}(c_1(b_0)+b_1)}\prod_{\ell
      =3}^k F_{n_\ell, j_\ell}(y) \mathrm{d}y.
  \]
  We may iterate this procedure $k-1$ times to subsequently pull
  out all factors from the integral. After this we end up with (the functions $c_2,\ldots, c_{k-1}$ are chosen in accordance with \eqref{eq:covering})
  
  \begin{equation}\label{eq:JJ}
    \begin{split}
      J\le & \sum_{b_0=0}^{aG_{n_0}-1}m(j_1,b_0)
      \sum_{b_1=0}^{g_1+1}m(j_2,c_1(b_0)+b_1) \cdots
      \sum_{b_{k-1}=0}^{g_{k-1}+1}m(j_{k},c_{k-1}(b_0,\ldots, b_{k-2})+b_{k-1}) \\
      &
      \cdot\int_{I_{n_{k-1}}(c_{k-1}(b_1,\ldots,b_{k-2})+b_{k-1})}\mathrm{d}y \\
      =& \frac{1}{aG_{n_{k-1}}}\sum_{b_0=0}^{aG_{n_0}-1}m(j_1,b_0)
      \sum_{b_1=0}^{g_1+1}m(j_2,c_1(b_0)+b_1) \cdots
      \sum_{b_{k-1}=0}^{g_{k-1}+1}m(j_{k},c_{k-1}(b_0,\ldots,
      b_{k-2})+b_{k-1}).
    \end{split}
  \end{equation}
  Thus we have to deal with sums of the form
  \[
    K_\ell=\sum_{b=0}^{g_\ell +1} m(j_{\ell+1},c+b) \qquad(c\in\Z,\, \ell\in\{1,\ldots, k-1\}).
  \]
  We distinguish two cases according to whether $j_\ell=n_\ell-n_{\ell-1}=1$
  or not. If $j_\ell=n_\ell - n_{\ell-1}=1$ then, for $\ell$ sufficiently
  large, $g_\ell = a$ (because the dominant root satisfies \eqref{eq:alpharange}) and, hence, for $j=j_{\ell+1}$ we get
  \begin{equation}\label{eq:K1est}
    \begin{split}
      K_\ell&= \sum_{b=0}^{a+1} m(j,c+b) \le a m(j) + 2 \max_{0\le b<a}m(j,b) \le am(j)+2a_j \le a(m(j)+2)\\
      &\le \frac{G_{n_\ell}}{G_{n_{\ell-1}}}(m(j)+2).
    \end{split}
  \end{equation}
  If $j_\ell = n_\ell - n_{\ell-1}>1$ for $\ell$ sufficiently large,
  $g_\ell \ge a^2$ and we may write $g_\ell+1 = ha+r$ with $h\ge a$ and
  $0\le r < a$ yielding (for $j=j_{\ell+1}$)
  \begin{align*}
      K_\ell&=\sum_{t=0}^{h-1} \sum_{u=0}^{a-1} m(j,c+t a+u)+ \sum_{u=0}^{r} m(j,c+ha+u) \\
      &\le ha\,m(j)+(r+1)a_j \le ha(m(j)+1) \le
      \frac{G_{n_\ell}}{G_{n_{\ell-1}}}(m(j)+1).
  \end{align*}
Inserting this in \eqref{eq:JJ} for all sufficiently large $\ell$ and observing that 
$
\sum_{b_0=0}^{aG_{n_0}-1}m(j_1,b_0) \le %a G_{n_0}m(j_1)\le 
a^2G_{n_0}
$
we get the result.
\end{proof}

\begin{prop}\label{prop:1norm}
  Let $G=(G_j)_{j\ge 1}$ be a linear recurrence base with
  characteristic polynomial $X^{d}-a_1X^{d-1}-\cdots-a_{d-1}X-a_d$
  satisfying the conditions of Definition~\ref{def:lrb}. Fix
  $k,n\in \mathbb{N}$ and let $j_1,j_2,\ldots,j_k\in\mathcal{I}$. Then
  \[
    \int_0^1 \prod_{\ell=1}^k\lvert
    A_{n-\sum_{r=1}^{\ell-1}j_r,j_\ell}(y,\beta)\rvert
    \mathrm{d}y \ll (m+2)^k,
  \]
  where $m=m_G$ is as in \eqref{eq:ma0}.
\end{prop}

\begin{proof}
  Using \eqref{eq:Af} 
  %and the upper bound $F_{k,i}$ of $f_{k,i}$, 
  the
  product may be rewritten as
  \begin{equation*}%\label{eq:Iprod2}
      \prod_{\ell=1}^k\lvert
      A_{n-\sum_{r=1}^{\ell-1}j_r,j_\ell}(y,\beta)\rvert
      =\prod_{\ell=1}^kf_{n-\sum_{r=1}^{\ell-1} j_r,j_\ell}(y)
      %\ll\prod_{\ell=1}^kF_{n-\sum_{r=1}^{\ell-1} j_r,j_\ell}(y).
\end{equation*}
The last product satisfies the conditions of Proposition~\ref{prop:Fprod} and we obtain our result by applying this proposition.
%Thus writing
%\[
%n_\ell = n-\sum_{r=1}^{k-\ell}j_r \quad\hbox{for } 1\le \ell\le k \quad \hbox{and}\quad i_\ell = j_{k+1-\ell}\quad\hbox{for $1\le \ell\le k$},
%\]
%we can use Proposition~\ref{prop:Fprod} to obtain
%\begin{equation}\label{eq:K2}\begin{split}
%\int_0^1\prod_{\ell=1}^k\lvert
%A_{n-\sum_{r=1}^{\ell-1}j_r,j_\ell}(y,\beta)\rvert
%\mathrm{d}y
%\ll \int_0^1 \prod_{\ell=1}^k F_{n_\ell, i_\ell}(y) \mathrm{d}y
%\ll (m+2)^k. \qedhere 
%\end{split}
%\end{equation}
\end{proof}

We now state our estimate for the $1$-norm of $S_n(\cdot,\beta)$. Note that in the following result the estimate $\Vert S_{n}(\cdot,\beta)  \Vert_1 \ll \alpha^{\frac{n}{2}}$ is derived by easy general arguments (as in the classical case, see \cite[Lemme~7]{MR:10} and \cite[Remarks after Th\'eor\`eme~2 and the beginning of Section IV]{Fouvry-Mauduit:96}).

\begin{prop}\label{prop:S-1norm}
Let $G=(G_j)_{j\ge 0}$ be a linear recurrence base with characteristic polynomial $X^{d}-a_1X^{d-1}-\cdots-a_{d-1}X-a_d$ satisfying the conditions of Definition~\ref{def:lrb} and let $\alpha$ be its dominant root. Then
\[
\int_0^1\lvert S_{n}(y,\beta)\rvert\mathrm{d}y \ll \min\{\alpha^{\frac12},(m+3)\}^n,
\]
where $m=m_G$ is as in \eqref{eq:ma0}. 
\end{prop}

\begin{proof}
We first show that $\int_0^1\lvert S_{n}(y,\beta)\rvert\mathrm{d}y \ll \alpha^{\frac{n}{2}}$. As in \cite[Lemme~7]{MR:10}, this immediately follows by applying the Cauchy-Schwarz inequality, Parseval's identity. Indeed, using \eqref{eq:recasympt} we obtain
\[
\int_0^1\lvert S_{n}(y,\beta)\rvert\mathrm{d}y
\le 
\bigg(\int_0^1\lvert S_{n}(y,\beta)\rvert^2\mathrm{d}y\bigg)^{\frac{1}{2}}
 = 
\bigg(
 \int_0^1\bigg\vert \sum_{k<G_n}e\left(\beta s_G(k))e(y k\right)\bigg\vert^2\mathrm{d}y
 \bigg)^{\frac{1}{2}}\ll \alpha^\frac{n}{2}.
\]

It remains to prove that $\int_0^1\lvert S_{n}(y,\beta)\rvert\mathrm{d}y \ll (m+3)^n$,
   In view of \eqref{eq:Iprodbis} we have to deal with the cardinality
   of $J_k(d)$ before we can apply Proposition~\ref{prop:1norm}. To this
   matter let 
   %$C_{k,d}(n)$ be the set of compositions of $n$ in $k$ parts, where each part is a positive integer bounded by $d$, {\it  i.e.},
 \[
 C_{k,d}(n)=\left\{(j_1,\ldots,j_k) \in \{1,\ldots, d\}^k\;:\; n=j_1+\cdots + j_k\right\}.
 \]
 It easy to see that $\#C_{k,d}(n) \le \binom{n-1}{k-1}$ (there exist
 exact formulas, see {\it e.g.}~Abramson~\cite{Abramson:76}). Since
 $ J_k(d) \subset \bigcup_{j=1}^{d}C_{k,d}(n-j) $ we gain
 $\# J_k(d) \ll \binom{n}{k}$.  Using this in \eqref{eq:Iprodbis}
 together with Proposition~\ref{prop:1norm} and the binomial theorem yields
\begin{align*}
    \lVert S_{n}(y,\beta)\rVert_1
    &\leq\sum_{k=1}^{n-d}\sum_{(j_1,\ldots,j_k)\in
      J_k(d)} \int_0^1\prod_{\ell=1}^k\lvert
    A_{n-\sum_{r=1}^{\ell-1}j_r,j_\ell}(y,\beta)\rvert \cdot
    \lvert S_{n-j_1-\cdots-j_{k}}(y,\beta)\rvert \mathrm{d}y
\\
    & \ll
    \sum_{k=1}^{n-d}\sum_{(j_1,\ldots,j_k)\in
      J_k(d)} \int_0^1\prod_{\ell=1}^k\lvert
    A_{n-\sum_{r=1}^{\ell-1}j_r,j_\ell}(y,\beta)\rvert  \mathrm{d}y \\
&\ll  \sum_{k=1}^{n-d}\sum_{(j_1,\ldots,j_k)\in
      J_k(d)}(m+2)^k \ll  \sum_{k=1}^{n-d} \binom{n}{k} (m+2)^k \ll (m+3)^n. \qedhere
    \end{align*}
\end{proof}

As mentioned at the beginning of Section \ref{sec:estimates} we also need the $1$-norm of the derivative of $S_n(y,\beta)$ with respect to the first variable.

\begin{prop}\label{prop:S'-1norm2}
  Let $G=(G_j)_{j\ge 0}$ be a linear recurrence base with
  characteristic polynomial $X^{d}-a_1X^{d-1}-\cdots-a_{d-1}X-a_d$
  satisfying the conditions of Definition~\ref{def:lrb} and let
  $\alpha$ be its dominant root. Then
  \[
    \int_0^1\lvert \frac{\partial S_{n}}{\partial y}(y,\beta)\rvert\mathrm{d}y \ll \alpha^n \min\{ \alpha^{\frac{n}{2}},(m+3)^n \},
  \]
  where $m=m_G$ is as in \eqref{eq:ma0}.
\end{prop}

\begin{proof}
Again we use the Cauchy-Schwarz inequality, Parseval identity, and \eqref{eq:recasympt} to show that
\begin{align*}
\int_0^1\bigg\vert  \frac{\partial S_{n}}{\partial y}(y,\beta)\bigg\vert\mathrm{d}y
&\le 
\bigg(\int_0^1\bigg\vert  \frac{\partial S_{n}}{\partial y}(y,\beta)\bigg\vert^2\mathrm{d}y\bigg)^{\frac{1}{2}}\\
 &= 
\bigg(
 \int_0^1\bigg\vert \sum_{k<G_n}2\pi k e\left(\beta s_G(k))e(y k\right)\bigg\vert^2\mathrm{d}y\bigg)^{\frac{1}{2}}
 =\bigg(\sum_{k<G_n} (2\pi k)^2 \bigg)^{\frac12}
 \ll\alpha^\frac{3n}{2}.
\end{align*}

It remains to show that  $\int_0^1\lvert \frac{\partial S_{n}}{\partial y}(y,\beta)\rvert\mathrm{d}y \ll \alpha^n(m+3)^n$. Using Equation (\ref{eq:Iprodbis}) we obtain for the $1$-norm of the
  derivative of $S_n$ that
  \[ 
  \bigg\vert
  \frac{\partial S_{n}}{\partial y}(y,\beta)
  \bigg\vert
  \ll
    \sum_{k=1}^{n}\sum_{(j_1,\ldots,j_k)\in J_k(d)}\sum_{1\leq i\leq
      k}G_{n-\sum_{r=1}^ij_r}\prod_{\substack{\ell=1\\ \ell\neq
        i}}^k\lvert
    A_{n-\sum_{r=1}^{\ell-1}j_r,j_\ell}\left(y,\beta\right)\rvert.
  \]
  Since $G_k\ll\alpha^k$ and  $\lvert A_{k,j}(y,\beta)\rvert\leq a_j<\alpha$ we obtain
  \[    \lVert\frac{\partial S_{n}}{\partial
    y}(y,\beta)\rVert_1
    \ll\sum_{k=1}^{n}\sum_{(j_1,\ldots,j_k)\in J_k(d)}\sum_{1\leq i\leq
      k}\alpha^{n-\sum_{r=1}^{i}j_r}\alpha^{\sum_{r=1}^ij_r}\int_0^1\prod_{\ell=i+1}^{k}\lvert
    A_{n-\sum_{r=1}^{\ell-1}j_r,j_\ell}\left(y,\beta\right)\rvert\mathrm{d}y.
  \]
  Now an application of Proposition~\ref{prop:1norm} yields
  \begin{gather*}
    \lVert\frac{\partial S_{n}}{\partial
    y}(y,\beta)\rVert_1
    \ll\sum_{k=1}^{n}\sum_{(j_1,\ldots,j_k)\in J_k(d)}\alpha^{n}\sum_{1\leq
      i\leq k}(m+2)^{k-i}
    \ll\alpha^n(m+3)^n,
  \end{gather*}
  where we once more used that $J_k(d)\ll\binom{n}{k}$ and the binomial theorem. 
\end{proof}

\begin{rem}\label{rem:better}
  If we deal with particular cases of linear recurrences it is
  possible to improve the estimate in Propositions~\ref{prop:S-1norm} and~\ref{prop:S'-1norm2}
  slightly by the following consideration. Let $\varepsilon >0$ be
  arbitrary. Then there is $N\in\N$ such that
  $ \left\lfloor {G_n}/{G_{n-1}}\right\rfloor + 1 - {G_n}/{G_{n-1}}
  \ge \lfloor \alpha \rfloor +1-\alpha - \varepsilon := u$ holds for all $n\ge N$. Let $r$ be
  the smallest positive integer satisfying $r^{-1} < u$. Since
  $I_n(c)$ is an interval of length $1/aG_n$ for each $b$ there is
  $t \in R_r:=\{0, r^{-1}, \ldots, (r-1)r^{-1}\}$ and $c\in \mathbb{N}$
  such that
\[
I_{n-1}(b) \subset I_{n}(c+t)\cup I_{n}(c+1+t) \cup \ldots \cup I_{n}(c+ \left\lfloor {G_n}/{G_{n-1}}\right\rfloor +t).
\]
We use this instead of \eqref{eq:covering} in the proof of Proposition~\ref{prop:Fprod} whenever $n_\ell-n_{\ell-1} = 1$ and replace the maxima $m(a,b)$ by 
\begin{equation*}\label{eq:mabt}
m_G^{(t)}(j,b)=m^{(t)}(j,b)=\sup_{y\in (\frac {b+t}{a},\frac{b+t+1}{a})}\left\vert\frac{\sin \pi a_j y}{\sin \pi y}\right\vert \qquad(j\in\mathcal{I}, b \in\Z)
\end{equation*}
in these cases. With these modifications we obtain (since we get a better estimate for $K_\ell$ in \eqref{eq:K1est})
\begin{equation}\label{eq:better21}
\int_0^1 \prod_{\ell=1}^k f_{n_\ell, i_\ell}(y) \mathrm{d}y \ll (m^{(r)}+1)^k
\end{equation}
with 
\begin{equation}\label{eq:m(r)}
m_G^{(r)}=m^{(r)}=\max_{j\in \mathcal{I},\,t\in R_r} m^{(t)}(j), 
\end{equation}
where $m_G^{(t)}(j)= \frac1a\sum_{b=0}^{a-1} m_G^{(t)}(j,b)$ for $j\in\mathcal{I}$.
Applying \eqref{eq:better21} in Proposition~\ref{prop:1norm} instead of Proposition~\ref{prop:Fprod} we gain 
\begin{equation}\label{prop1improved}
\int_0^1\lvert S_{n}(y,\beta)\rvert\mathrm{d}y \ll (m^{(r)}+2)^n \quad \hbox{and} \quad  \int_0^1\lvert \frac{\partial S_{n}}{\partial y}(y,\beta)\rvert\mathrm{d}y \ll \alpha^n(m^{(r)}+2)^n.
\end{equation}
\end{rem}

\subsection{The maximum norm of sums related to $S_n(y,\beta)$}
The maximum norm of $S_n(y,\frac{r}{s})$ has been estimated by Lamberger and Thuswaldner~\cite{Lamberger-Thuswaldner:03}. However, for our purposes we require a variant of their estimate. To establish this variant we need some notation and some results from~\cite{Lamberger-Thuswaldner:03}. Let $G=(G_j)_{j\ge 0}$ be a linear recurrence base with characteristic polynomial $X^{d}-a_1X^{d-1}-\cdots-a_{d-1}X-a_d$ satisfying the conditions of Definition~\ref{def:lrb}. Fix $r,s\in\N$ and $y\in\R$ in a way that $\gcd(a_1+\cdots+a_d-1,s)=1$ and $r \not\equiv 0 \pmod{s}$. According to~\cite{Lamberger-Thuswaldner:03}  by iterating \eqref{eq:Srec} in an appropriate way we can obtain a recurrence 
\begin{equation}\label{eq:Srec2}
S_n\Big(y,\frac{r}{s}\Big)=\sum_{j\le D}B_{n,j}\Big(y,\frac{r}{s}\Big)S_{n-j}\Big(y,\frac{r}{s}\Big) \qquad(n\ge D)
\end{equation}
of order $D >d$ with coefficient functions $B_{n,j}(y,\frac{r}{s})$ having the  following properties (for $a_1>1$ this recurrence is written explicitly in \cite[Equation~(5)]{Lamberger-Thuswaldner:03} and for $a_1=1$ it is written in \cite[Equation~(12)]{Lamberger-Thuswaldner:03}; however, we do not need these formulas here):

By \cite[Proposition~1]{Lamberger-Thuswaldner:03} there exist $b_1,\ldots, b_D\in \R$ with $b_j \ge |B_{n,j}(y,\frac{r}{s})|$ for all $1\le j\le D$ and all  $n\in\N$ such that the linear recurrent sequence  
\begin{equation}\label{eq:TN0}
T_{n+D} = \sum_{j=1}^D b_j T_{n+D-j}  \qquad (n\ge 0)
\end{equation}
satisfies
\begin{equation}\label{eq:TN1}
\Big|S_n\Big(y,\frac{r}{s}\Big)\Big| < T_n  \qquad (n \in \N)
\end{equation}
for certain initial values $T_0,\ldots, T_{D-1} \in \R_{>0}$. Moreover, from \cite[Section~4.1]{Lamberger-Thuswaldner:03} we see that there is a constant $\lambda=\lambda(G,s)<1$ such that
\begin{equation}\label{eq:TN2}
\alpha^{\lambda n}  \ll  T_n \ll \alpha^{\lambda n} \qquad (n\in \N).
\end{equation}

We also need an analog of $J_k(n_0)$ from \eqref{eq:iterate2}. For $D\leq n_0< n$ and $1\leq k< n$ let 
%$K_k(n_0)$ be the set of iterations $(j_1,\ldots,j_k) \in\{1,\ldots, D\}^k$ that step over $n_0$ with the last entry $j_k$, \textit{i.e.},
\begin{equation}\label{eq:iterate2B}
K_k(n_0)=\left\{\mathbf{j}=(j_1,\ldots,j_k)\in\{1,\ldots,D\}^k\colon
  n-\sum_{\ell=1}^{k-1} j_\ell>n_0\geq n-\sum_{\ell=1}^{k}j_\ell\right\}.
\end{equation}

\begin{prop}\label{prop:S-maxA}
Let $G=(G_j)_{j\ge 0}$ be a linear recurrence base with characteristic polynomial $X^{d}-a_1X^{d-1}-\cdots-a_{d-1}X-a_d$ satisfying the conditions of Definition~\ref{def:lrb}. Let $n,r,s\in\N$ and $y\in \R$ be given in a way that $\gcd(a_1+\cdots+a_d-1,s)=1$ and $r \not\equiv 0 \pmod{s}$. Then for each $n_1\in\{D,\ldots, n-1\}$ we have  
\[   
 \sum_{k=1}^{n-n_1}\sum_{(j_1,\ldots,j_k)\in
    K_k(n_1)}\prod_{\ell=1}^k\lvert
    B_{n-\sum_{r=1}^{\ell-1}j_r,j_\ell}\left(y,\frac{r}{s}\right)\rvert \ll \alpha^{\lambda(n-n_1)},
 \]
where $\lambda=\lambda(G,s)<1$ and the implied constant depends only on the linear recurrence base $G$ and the integer $s$.
\end{prop}

\begin{proof}
Let $(T_n)_{n\ge 0}$ be the linear recurrent sequence given by \eqref{eq:TN0} (with initial values satisfying \eqref{eq:TN1}). By the definition of $b_1,\ldots,b_D$, and $K_k(n_1)$ we get, using \eqref{eq:TN2}, that
\begin{align*}\label{eq:IprodbisB}
 \sum_{k=1}^{n-n_1}\sum_{(j_1,\ldots,j_k)\in
    K_k(n_1)}\prod_{\ell=1}^k\lvert
    B_{n-\sum_{r=1}^{\ell-1}j_r,j_\ell}\left(y,\frac{r}{s}\right)\rvert &\le
    \sum_{k=1}^{n-n_1}\sum_{(j_1,\ldots,j_k)\in
      K_k(n_1)} \prod_{\ell=1}^k
   b_{j_\ell}  \\
 & \ll
  \alpha^{-\lambda n_1} \sum_{k=1}^{n-n_1}\sum_{(j_1,\ldots,j_k)\in
      K_k(n_1)} \left(\prod_{\ell=1}^k
   b_{j_\ell} \right) T_{n-\sum_{r=1}^{k}j_r}\\
&   =  \alpha^{-\lambda n_1} T_n \\
& \ll \alpha^{\lambda(n-n_1)}. \qedhere
\end{align*}
\end{proof}

We mention that, analogously to \eqref{eq:Iprodbis} we get the estimate
\begin{equation}\label{eq:IprodbisB}
  \begin{split}
    \lvert S_{n}(y,\beta)\rvert
    &\leq\sum_{k=1}^{n-n_0}\sum_{(j_1,\ldots,j_k)\in
      K_k(n_0)}\prod_{\ell=1}^k\lvert
    B_{n-\sum_{r=1}^{\ell-1}j_r,j_\ell}(y,\beta)\rvert \cdot
    \lvert S_{n-\sum_{r=1}^{k}j_r}(y,\beta)\rvert
  \end{split}
\end{equation}
for each $D\le n_0<n$.

\subsection{Upper bounds for $m_G$}\label{sec:m}
Let $G$ be a linear recurrence base as in Definition~\ref{def:lrb} and let $\alpha$ be the dominant root of the characteristic polynomial $X^d-a_1X^{d-1}-\cdots - a_{d-1}X - a_d$ of $G$. According to Proposition~\ref{prop:S-1norm} the $1$-norm of $S_n(\cdot,\beta)$ can be easily bounded by $\alpha^{n/2}$ by using Cauchy's inequality followed by Parseval's identity. However, often Proposition~\ref{prop:S-1norm} is of use only if this bound can be sharpened (and the same holds for Proposition~\ref{prop:S'-1norm2}). In particular, in view of \eqref{eq:thetabound} it will turn out that it is desirable to get $m_G+3 \le \alpha^{0.4886061}$, where the quantity $m_G$ is defined in \eqref{eq:ma0}. Such a sharpened bound is needed for instance in the proof of Corollary~\ref{cor1}.  Unfortunately, we are not able to achieve such an improvement for all $G$ satisfying the conditions of Definition~\ref{def:lrb}, however, we can achieve it if the coefficient $a_1$ is large enough. To get the treshhold value for $a_1$ as low as possible we will now study $m_G$ in some detail. Since the dependence of $m_G$ on the linear recurrence base $G$ will be crucial we keep the index $G$ in $m_G$ as well as in $m_G(j)$ (defined in \eqref{eq:ma}) throughout this section.

We start with the following estimate which is related to estimates established in \cite[Section~VI]{fouvry_mauduit1996:methodes_de_crible}.

\begin{lem}\label{lem:mest}
Let $G$ be a linear recurrence base as in Definition~\ref{def:lrb} and let $\alpha$ be the dominant root of the characteristic polynomial $X^d-a_1X^{d-1}-\cdots - a_{d-1}X - a_d$ of $G$. Let $m_G=\max_{j\in \mathcal{I}} m_G(j)$ with $m_G(j)$ as in \eqref{eq:ma}. Then for $a_1 \ge 3$ we have 
\begin{equation}\label{eq:lemmeq}
m_G \le 2 + \frac{2}{a_1 \sin \frac{\pi}{a_1}} - \frac2\pi \log\tan\frac\pi {2a_1}.
\end{equation}
This implies that $m_G  \ll \log a_1 \ll \log \alpha$ for large $a_1$. 
\end{lem}  

\begin{proof}
For convenience we set $a=a_1$. Fix $G$ in a way that $a\ge 3$ and set $I(b)=(\frac ba,\frac{b+1}a)$ for $b\in\mathbb{Z}$. First observe that, since $a_j\le a$ for $j\in\mathcal{I}$,
\[
m_G(j) \le 2 + \frac1a\sum_{b=1}^{a-2}\sup_{y\in I(b)}\frac{1}{\sin \pi y}.
\]
If $a\equiv1\pmod{2}$ we obtain
\begin{equation}\label{eq:evenm1}
m_G(j) \le 2 + \frac2a\sum_{b=1}^{\frac{a-3}2}\sup_{y\in I(b)}\frac{1}{\sin \pi y} + \frac1a \sup_{y\in I((a-1)/2)}\frac{1}{\sin \pi y} \le 2 + \frac2a\sum_{b=1}^{\frac{a-1}2}\sup_{y\in I(b)}\frac{1}{\sin \pi y}
\end{equation}
for each interval $I(b)$ in the rightmost sum the supremum of $\frac{1}{\sin \pi y}$ is located on the left end point of $I(b)$. Thus 
\[
\begin{split}
\sum_{b=1}^{\frac{a-1}2}\sup_{y\in I(b)}\frac{1}{\sin \pi y} &= \sum_{b=1}^{\frac{a-1}2}\frac{1}{\sin  \frac {\pi b}a}  
\\
&\le \frac{1}{\sin  \frac {\pi}a} + \int_{1}^{\frac{a-1}{2}}\frac{dx}{\sin  \frac {\pi x}a}
\\
& = \frac{1}{\sin  \frac {\pi}a} + \frac{a}{\pi} \log\frac{\tan(\frac\pi4 - \frac\pi{4a})}{\tan\frac\pi{2a}}
\\
& \le \frac{1}{\sin  \frac {\pi}a} - \frac{a}{\pi} \log\tan\frac\pi{2a}.
\end{split}
\]
Inserting this in \eqref{eq:evenm1} we arrive at 
\begin{equation}\label{eq:mas-bound}
m_G(j) \le 2 + \frac{2}{a \sin \frac\pi a} - \frac2\pi \log\tan\frac\pi a.
\end{equation}
If $a\equiv0\pmod{2}$ we obtain
\begin{equation}\label{eq:evenm2}
m_G(j)  \le 2 + \frac2a\sum_{b=1}^{\frac a2-1}\sup_{y\in I(b)}\frac{1}{\sin \pi y}.
\end{equation}
Similar to the case of odd $a$ we now gain
\[
\begin{split}
  \sum_{b=1}^{\frac a2-1}\sup_{y\in I(b)}\frac{1}{\sin \pi y} &=
  \sum_{b=1}^{\frac a2-1}\frac{1}{\sin \frac {\pi b}a} 
  \\
  &\le
  \frac{1}{\sin \frac {\pi}a} + \int_{1}^{\frac a2 -1}\frac{dx}{\sin
    \frac {\pi x}a}
  \\
  & = \frac{1}{\sin \frac {\pi}a} + \frac{a}{\pi}
  \log\frac{\tan(\frac\pi4 - \frac\pi{2a})}{\tan\frac\pi{2a}}
  \\
  & \le \frac{1}{\sin \frac {\pi}a} - \frac{a}{\pi}
  \log\tan\frac\pi{2a}.
\end{split}
\]
Inserting this in \eqref{eq:evenm2} we get \eqref{eq:mas-bound} also
in this case. The estimate \eqref{eq:lemmeq} now follows from the
definition of $m_G$.
% because $a_j\le a$ holds for all $j\in\mathcal{I}$. 
The asymptotic result is an immediate consequence
of \eqref{eq:lemmeq} since $\tan x \sim x$ for $x\to 0$, and $a\le \alpha$ holds by \eqref{eq:alpharange}.
\end{proof}

The above result immediately implies that
$m_G+3 < \alpha^{0.4886061}$ holds for all $a_1 \ge 72$. By
calculating $m_G$ directly (the suprema have to be approximated
numerically which has been done using {\tt Mathematica}) we get that
this even holds for $a_1 \ge 59$. Thus Proposition~\ref{prop:S-1norm}  and Proposition~\ref{prop:S'-1norm2} immediately imply the
following lemma.

\begin{lem}\label{lem:estallgG}
  Let $(G_j)_{j\ge0}$ be a linear recurrence base satisfying the conditions of Definition~\ref{def:lrb} whose
  characteristic polynomial is given by $X^d-a_1X^{d-1}-\cdots- a_{d-1}X-a_d$
  and has dominant root $\alpha$.  If $a_1 \ge 59$
  then 
  \[
  \int_0^1\lvert S_{n}(y,\beta)\rvert\mathrm{d}y \ll\alpha^{n\eta} \quad\hbox{ and }\qquad  \int_0^1\lvert \frac{\partial S_{n}}{\partial y}(y,\beta)\rvert\mathrm{d}y \ll \alpha^{n(1+\eta)}
  \] 
  hold for some explicitly computable
  $\eta < 0.4886061$.
\end{lem}

For the special family $(G_j)_{j\ge 0}$ with $G_{j+2}=aG_{j+1}+G_j$ we
use Remark~\ref{rem:better} to get this result for even smaller values of  $a_1$. Indeed, if
$a \ge 40$ we may choose $r=2$ in this remark and, again using {\tt
  Mathematica}, we can calculate the quantity $m_G^{(2)}$ defined in \eqref{eq:m(r)} for $40 \le a \le
58$. Since $m_G^{(2)}+2 < \alpha^{0.4886061}$ holds for all $a \ge 40$, the
estimate in \eqref{prop1improved} yields the following result.

\begin{lem}\label{lem:estGspecial}
Let $(G_j)_{j\ge0}$ be a linear recurrence base whose characteristic polynomial is given by $X^2-aX-1$ and has dominant root $\alpha$. If $a \ge 40$ then
 \[
  \int_0^1\lvert S_{n}(y,\beta)\rvert\mathrm{d}y \ll\alpha^{n\eta} \quad\hbox{ and }\qquad  \int_0^1\lvert \frac{\partial S_{n}}{\partial y}(y,\beta)\rvert\mathrm{d}y \ll \alpha^{n(1+\eta)}
  \] 
  hold for some explicitly computable $\eta < 0.4886061$. 
\end{lem}

\section{Estimates of the $1$-norm for smaller values of $a_1$}\label{sec:smaller-traces}

\subsection{Blocking}
As mentioned at the beginning of Section~\ref{sec:m}, in order to derive results on almost primes we need to get good bounds for the $1$-norm of $S_n(\cdot,\beta)$ and of its derivative. To obtain such good estimates also for smaller coefficients $a_1$, instead of taking suprema after each step of the recurrence, we deal with ``blocks'' or ``windows'' of ``width'' $w$ and take the suprema after each $w$-th iteration. To keep things as simple as possible we only do this for recurrences having characteristic polynomial $X^2-aX-1$ for some $a\ge 1$ (it should then be clear how to treat the general case). Thus in the present section $G=(G_n)$ is defined by
\begin{equation}\label{eq:specialrec}
G_{n+2}=aG_{n+1} + G_n \qquad (n\ge 0)
\end{equation}
with $G_0=1$ and $G_1\ge a+1$. In this case \eqref{eq:Srec} becomes
\begin{equation}\label{eq:recspecialcase}
S_n(y,\beta) = A_{n,1}(y,\beta)S_{n-1}(y,\beta) + A_{n,2}(y,\beta)S_{n-2}(y,\beta) \qquad(n\ge 2).
\end{equation}
Now set $A^{(1)}_{n,j}(y,\beta)=A_{n,j}(y,\beta)$ for $j\in\{1,2\}$ and recursively define
\[
\begin{split}
A_{n,\ell}^{(\ell)}(y,\beta) &= A_{n,\ell-1}^{(\ell-1)}(y,\beta)\cdot A_{n-\ell+1,1}(y,\beta)+A_{n,\ell}^{(\ell-1)}(y,\beta)\quad\hbox{and}\quad \\
A_{n,\ell+1}^{(\ell)}(y,\beta)&=A_{n,\ell-1}^{(\ell-1)}(y,\beta)\cdot A_{n-\ell+1,2}(y,\beta).
\end{split}
\]
If we iterate \eqref{eq:recspecialcase} appropriately we obtain
\begin{equation}\label{eq:Swindow}
S_n(y,\beta) = A^{(w)}_{n,w}(y,\beta)S_{n-w}(y,\beta) + A^{(w)}_{n,w+1}(y,\beta)S_{n-w-1}(y,\beta).
\end{equation}
Setting
\begin{equation}\label{eq:Jw}
J^{(w)}_{k}=\left\{ (j_1,\ldots,j_k) \in \{w,w+1\}^k \;:\; n-\sum_{\ell=1}^{k-1}j_\ell > w+1 \ge  n-\sum_{\ell=1}^{k}j_\ell \right\} \qquad \big(1\le k \le \lfloor n/w \rfloor\big)
\end{equation}
and iterating \eqref{eq:Swindow} we find in the same way as in the proof of Proposition\eqref{prop:S-1norm} that
\[
|S_n(y,\beta)| \ll \sum_{k=1}^{\lfloor n/w \rfloor} \sum_{(j_1,\ldots, j_k) \in J_k^{(w)}} \prod_{\ell=1}^k\left| A^{(w)}_{n-\sum_{r=1}^{\ell-1}j_r,j_\ell}(y,\beta)\right|.
\]
The functions $A_{n,j}^{(w)}(y,\beta)$ are exponential sums containing linear combinations of $G_{n-1},\ldots,G_{n-w-2}$ in the exponents. 
Moreover, their definition implies that $|A_{n,j}^{(w)}(y,\beta)|\le \alpha^j$ for $j\in\{w,w+1\}$. Thus as in the proof of Proposition~\ref{prop:S'-1norm2} we get
\begin{align*}
\bigg\vert
\frac{\partial S_{n}}{\partial y}(y,\beta)
\bigg\vert
&\ll
    \sum_{k=1}^{\lfloor n/w \rfloor}\sum_{(j_1,\ldots, j_k) \in J_k^{(w)}}\sum_{1\leq i\leq
      k}G_{n-\sum_{r=1}^ij_r}\prod_{\substack{\ell=1\\ \ell\neq
        i}}^k\lvert
    A_{n-\sum_{r=1}^{\ell-1}j_r,j_\ell}\left(y,\beta\right)\rvert \\
&\ll\alpha^n \sum_{k=1}^{\lfloor n/w \rfloor}\sum_{(j_1,\ldots, j_k) \in J_k^{(w)}}\sum_{1\leq i\leq
      k}\prod_{{\ell=i+1}}^k\lvert
    A_{n-\sum_{r=1}^{\ell-1}j_r,j_\ell}\left(y,\beta\right)\rvert.
 \end{align*}
Since for the recurrences in \eqref{eq:specialrec} the asymptotic estimate \eqref{eq:recasympt} can be strengthened to $G_\ell=c\alpha^\ell + \mathcal{O}(\alpha^{-\ell})$ for some $c>0$ we replace $G_\ell$ by $c\alpha^\ell$ in $A_{n,j}^{(w)}(y,\beta)$ and  call the resulting expression $\tilde A_{n,j}^{(w)}(y,\beta)$. Then $|\tilde A_{k,j}^{(w)}(y,\beta)- A_{k,j}^{(w)}(y,\beta)| \ll \alpha^{-k}$ for $j\in\{w,w+1\}$. Thus for each $\delta >0$ we have
\[
|S_n(y,\beta)| \ll \sum_{k=1}^{\lfloor n/w \rfloor} \sum_{(j_1,\ldots, j_k) \in J_k^{(w)}} \prod_{\ell=1}^k\left(\left| \tilde A^{(w)}_{n-\sum_{r=1}^{\ell-1}j_r,j_\ell}(y,\beta) \right| + \delta\right),
\]
where the implied constant depends on $\delta$. Obviously, an analogous estimate holds {\it mutatis mutandis} for $\vert
\frac{\partial S_{n}}{\partial y}(y,\beta)\vert$. Instead of the intervals $I_k(b)$ used in Section~\ref{sec:estimates} we now use the intervals
\[
I'_k(b)=\left[\frac{b-a\beta}{ac\alpha^k},\frac{b+1-a\beta}{ac\alpha^k}\right) \qquad(0\leq b<\lfloor ac\alpha^k \rfloor).
\]
(For large $k$, the intervals $I_k(b)$ and $I'_k(b)$ are almost the same.) Now we define
\begin{equation}\label{eq:Mwjb}
M_w(j,b) = \sup_{y\in I'_{n-1}(b)} \left|\tilde A^{(w)}_{n,j}(y,\beta)\right| + \delta \qquad(j\in\{w,w+1\},\, b\in \Z)
\end{equation}
and note that $M_w(j,b)$ does not depend on $n$. Indeed, by \eqref{eq:A} and the definition of $\tilde A^{(w)}_{n,j}(y,\beta)$ the variable $n$ in $\tilde A^{(w)}_{n,j}(y,\beta)$ occurs only in linear combinations of terms of the form $\alpha^{n-k}y$ for some $k$ depending only on $a_j$ and $w$. Thus $n$ cancels out if we insert the bounds of the interval $I'_{n-1}(b)$ for $y$. However, contrary to $m(j,b)$, the function $M_w(j,b)$ is in general \emph{not} periodic in $b$ (also note that, contrary to the definition of $m(j,b)$ we use $n-1$ instead of $n-j$ as index of $I'_{n-1}(b)$; this is because we want to split in \emph{finer} subintervals in each step than we did in Section~\ref{sec:1n}). Setting
\[
F_{n,j}^{(w)}(y) = M_w(j,b) \quad\hbox{for}\quad y\in I'_{n-1}(b)
\]
and integrating we gain
\begin{equation}\label{eq:zws}
\int_0^1|S_n(y,\beta)|\mathrm{d}y \ll \sum_{k=1}^{\lfloor n/w \rfloor} \sum_{(j_1,\ldots, j_k) \in J_k^{(w)}} \int_0^1\prod_{\ell=1}^k F^{(w)}_{n-\sum_{r=1}^{\ell-1}j_r,j_\ell}(y)\mathrm{d}y
\end{equation}
and
\begin{equation}\label{eq:zws2}
\int_0^1 \bigg\vert
\frac{\partial S_{n}}{\partial y}(y,\beta)
\bigg\vert
\ll \alpha^n\sum_{k=1}^{\lfloor n/w \rfloor}\sum_{(j_1,\ldots, j_k) \in J_k^{(w)}}\sum_{1\leq i\leq
      k}
    \int_0^1\prod_{\ell=i+1}^k F^{(w)}_{n-\sum_{r=1}^{\ell-1}j_r,j_\ell}(y)\mathrm{d}y.
\end{equation}
Writing
\begin{equation}\label{eq:nelliell}
n_\ell = n-\sum_{r=1}^{k-\ell}j_r \quad\hbox{for } 0\le \ell\le k \quad \hbox{and}\quad i_\ell = n_\ell-n_{\ell-1} = j_{k+1-\ell}\quad\hbox{for $1\le \ell\le k$}
\end{equation}
we now consider the integrals $\int_0^1 \prod_{\ell=1}^k F^{(w)}_{n_\ell, i_\ell}(y) \mathrm{d}y$ in \eqref{eq:zws} and $\int_0^1 \prod_{\ell=1}^{k-i} F^{(w)}_{n_\ell, i_\ell}(y) \mathrm{d}y$ in \eqref{eq:zws2}.
From the
  definition of the intervals $I'_k(b)$ it is clear that each interval
  of the form $I'_{n_{\ell-1}}(b)$ can be covered by $\lfloor\alpha^{i_\ell}\rfloor + 2$
  adjacent intervals of the form $I_{n_{\ell}}(b')$. To be more precise, there is $c\in\N$ such that
  \begin{equation}\label{eq:istrich}
      I'_{n_{\ell-1}}(b) \subset I'_{n_\ell}(c)\cup I'_{n_\ell}(c+1) \cup \dots \cup I'_{n_\ell}(c+\lfloor\alpha^{i_\ell}\rfloor +1).
  \end{equation}
We can now argue in a similar way as in \eqref{eq:JJ} to gain 
(the functions $c_1,\ldots, c_{k-1}$ are chosen in accordance with \eqref{eq:istrich})
\begin{equation}\label{eq:m2spl}
\begin{split}
\int_0^1 \prod_{\ell=1}^k &F^{(w)}_{n_\ell, i_\ell}(y) \mathrm{d}y 
= \int_0^1 F^{(w)}_{n_1, i_1}(y)\prod_{\ell=2}^k F^{(w)}_{n_\ell, i_\ell}(y) \mathrm{d}y  \\
&\ll\sum_{b_0=0}^{\lfloor ac\alpha^{n_1-1}\rfloor}M_w({i_1},b_0)  \int_{I'_{n_1-1}(b_0)} \prod_{\ell=2}^k F^{(w)}_{n_\ell, i_\ell}(y) \mathrm{d}y\\
&\ll\sum_{b_0=0}^{\lfloor ac\alpha^{n_1-1}\rfloor}M_w({i_1},b_0)\sum_{b_1=0}^{\lfloor\alpha^{i_2}\rfloor+1}M_w(i_2,c_1(b_0)+b_1)\int_{I'_{n_2-1}(c_1(b_0)+b_1)}\prod_{\ell=3}^k F^{(w)}_{n_\ell, i_\ell}(y) \mathrm{d}y\\
&\ll \frac{1}{ac\alpha^{n-1}} \sum_{b_0=0}^{\lfloor ac\alpha^{n_1-1}\rfloor}M_w({i_1},b_0) \times \\
&\qquad
\sum_{b_1=0}^{\lfloor\alpha^{i_2}\rfloor+1}M_w(i_2,c_1(b_0)+b_1)
\cdots
\sum_{b_{k-1}=0}^{\lfloor\alpha^{i_k}\rfloor+1}M_w(i_{k},c_{k-1}(b_0,\ldots, b_{k-2})+b_{k-1}).
\end{split}
\end{equation}
Let
\begin{equation}\label{eq:Mwrdef}
M_w(r) = \sup_{q\in\mathbb{Z}} \sum_{b=0}^{\lfloor\alpha^r \rfloor +1} M_w(r,b+q) \qquad(r\in\{w,w+1\}).
\end{equation}
According to \eqref{eq:nelliell} we have $i_\ell=n_\ell -n_{\ell-1} \in \{w,w+1\}$. Thus, if 
\[
s=s(j_1,\ldots, j_k)=\#\{1\le \ell \le k \,:\, i_\ell = w\}=\#\{1\le \ell \le k \,:\, j_\ell = w\},
\] 
we have
\begin{equation}\label{eq:zzzw}
\int_0^1 \prod_{\ell=1}^k F^{(w)}_{n_\ell, i_\ell}(y) \mathrm{d}y \ll \frac{1}{\alpha^n} M_w(w)^sM_w(w+1)^{k-s}.
\end{equation}
%Assume that $M_w(r) \ge 1$ for $r\in\{w,w+1\}$ (which is not a big restriction) and 
If we iterate only $k-i-1$ times (instead of $k-1$ times) in \eqref{eq:m2spl} we get 
\begin{equation}\label{eq:zzzw2}
\begin{split}
\int_0^1 \prod_{\ell=1}^{k-i} F^{(w)}_{n_\ell, i_\ell}(y) \mathrm{d}y 
&\ll \frac{1}{\alpha^n} M_w(w)^{s(i)}M_w(w+1)^{k-i-s} \\
&\ll \frac{1}{\alpha^n} \max\{M_w(w),1\}^s\max\{M_w(w+1),1\}^{k-s}
\end{split}
\end{equation}
with $s(i)=\#\{1\le \ell \le k-i \,:\, i_\ell  = w\}$.

By the definition of $J_k^{(w)}$ in \eqref{eq:Jw} the sum $ws(j_1,\ldots, j_k) + (w+1)(k-s(j_1,\ldots, j_k))$ must be close to $n$ to make sure that $(j_1,\ldots, j_k) \in J_k^{(w)}$ holds. Using this fact and inserting \eqref{eq:zzzw} in \eqref{eq:zws} we finally gain
\begin{equation}\label{eq:intes}
\begin{split}
\int_0^1|S_n(y,\beta)|\mathrm{d}y &\ll \frac{1}{\alpha^n} \sum_{k=1}^{\lfloor n/w \rfloor} \sum_{(j_1,\ldots, j_k) \in J_k^{(w)}} M_w(w)^{s(j_1,\ldots, j_k)}M_w(w+1)^{k-s(j_1,\ldots, j_k)} \\
&\ll \frac{1}{\alpha^n} \sum_{s=1}^{\lfloor n/w \rfloor} \binom {{\lfloor n/w \rfloor}}s M_w(w)^sM_w(w+1)^{(n-ws)/(w+1)}\\
&= \frac{1}{\alpha^n} \sum_{s=1}^{\lfloor n/w \rfloor} \binom {{\lfloor n/w \rfloor}}s M_w(w)^s(M_w(w+1)^{w/(w+1)})^{{n}/{w} - s}\\
&\ll \frac{1}{\alpha^n} \big(M_w(w)+M_w(w+1)^{w/(w+1)}\big)^{n/w}.
\end{split}
\end{equation}
Since $k\le n$ inserting \eqref{eq:zzzw2} in \eqref{eq:zws2} in an analogous way we gain 
\begin{equation}\label{eq:newAbl}
\int_0^1\lvert \frac{\partial S_{n}}{\partial y}(y,\beta)\rvert\mathrm{d}y \ll n \big(\max\{M_w(w),1\}+\max\{M_w(w+1),1\}^{w/(w+1)}\big)^{n/w}.
\end{equation}
As mentioned in Section~\ref{sec:estimates} we want to get
\begin{equation}\label{eq:whatwewant}
  \int_0^1\lvert S_{n}(y,\beta)\rvert\mathrm{d}y \ll\alpha^{n\eta} \quad\hbox{ and }\qquad  \int_0^1\lvert \frac{\partial S_{n}}{\partial y}(y,\beta)\rvert\mathrm{d}y \ll \alpha^{n(1+\eta)}
\end{equation}
for some $\eta \le 0.4886061$. Thus in view of \eqref{eq:intes} and \eqref{eq:newAbl} we are left with finding bounds for the suprema $M_w(w)$ and $M_w(w+1)$ that imply
\begin{equation}\label{eq:MwMw+1}
\max\{M_w(w),1\}+\max\{M_w(w+1),1\}^{w/(w+1)} < \alpha^{1.4886061\cdot w}.
\end{equation}
We need \eqref{eq:whatwewant} for all $\beta\in[0,1)$ to get our results for arbitrary modules.

\subsection{Blocks of width two}
In this section we derive the estimate \eqref{eq:whatwewant} for the recurrences \eqref{eq:specialrec} with $15\le a\le 39$ by setting $w=2$ for the width of the block. Indeed, if we take $w=2$ we get from \eqref{eq:Swindow} that
\begin{equation*}
S_n(y,\beta) = A^{(2)}_{n,2}(y,\beta)S_{n-2}(y,\beta) + A^{(2)}_{n,3}(y,\beta)S_{n-3}(y,\beta).
\end{equation*}
with
\[
A^{(2)}_{n,2}(y,\beta) = A_{n,1}(y,\beta)A_{n-1,1}(y,\beta)+A_{n,2}(y,\beta) \quad\hbox{and}\quad A^{(2)}_{n,3}(y,\beta) = A_{n,1}(y,\beta)A_{n-1,2}(y,\beta).
\]
Inserting $w=2$ in \eqref{eq:Mwjb} yields therefore
\[
\begin{split}
M_2(2,b)-\delta&=\sup_{y\in I'_{n-1}(b)}|\tilde A_{n,1}(y,\beta) \tilde A_{n-1,1}(y,\beta)+\tilde A_{n,2}(y,\beta)| \\
&\le  \sup_{y\in I'_{n-1}(b)} \left|
\frac{\sin \pi a(\beta + c\alpha^{n-1}y)}{\sin \pi (\beta + c\alpha^{n-1}y)} 
\frac{\sin \pi a(\beta + c\alpha^{n-2}y)}{\sin \pi (\beta + c\alpha^{n-2}y)} 
\right| + 1\\
&=\sup_{y\in \left(\frac ba,\frac{b+1}a\right)}
\left|
\frac{\sin \pi a y}{\sin \pi y} 
\frac{\sin \pi a( \beta(1-\alpha^{-1}) + \alpha^{-1}y )}{\sin \pi ( \beta(1-\alpha^{-1})+ \alpha^{-1}y )} 
\right| + 1\\
\end{split}
\]
and
\[
\begin{split}
M_2(3,b)-\delta=\sup_{y\in I_{n-1}(b)}|\tilde A_{n,1}(y,\beta) \tilde A_{n-1,2}(y,\beta)| 
= \sup_{y\in \left(\frac ba,\frac{b+1}a\right)}
\left|
\frac{\sin \pi a y}{\sin \pi y}
\right|. \\
\end{split}
\]
Thus, setting $\delta'=(\lfloor \alpha^2 \rfloor +2) \delta$ we obtain from \eqref{eq:Mwrdef} that
\[
\begin{split}
M_2(2)-\delta' &= \sup_{q\in\mathbb{Z}} \sum_{b=0}^{\lfloor\alpha^2 \rfloor +1} (M_2(2,b+q) - \delta)\\
&\le \sup_{q\in\mathbb{Z}} \sum_{b=0}^{\lfloor\alpha^2 \rfloor +1} 
\bigg( 
\sup_{y\in \left(\frac {b+q}a,\frac{b+q+1}a\right)}
\left|
\frac{\sin \pi a y}{\sin \pi y} 
\frac{\sin \pi a( \beta(1-\alpha^{-1}) + \alpha^{-1}y )}{\sin \pi ( \beta(1-\alpha^{-1})+ \alpha^{-1}y )} 
\right| + 1 
\bigg).
\end{split}
\]
Since $\{1,\alpha^{-1}\}$ are rationally independent over $\Q$ we can omit the offset in the arguments of the sine functions in the last quotient without changing the supremum over $\Z$. This yields that
\begin{align*}
M_2(2)-\delta' &\le \sup_{q\in\mathbb{Z}} \sum_{b=0}^{\lfloor\alpha^2 \rfloor +1} \bigg( \sup_{y\in \left(\frac {b+q}a,\frac{b+q+1}a\right)}
\left|
\frac{\sin \pi a y}{\sin \pi y} 
\frac{\sin \pi a\alpha^{-1}y}{\sin \pi \alpha^{-1}y} 
\right| + 1 \bigg)
\\
&
=\sup_{q\in\mathbb{Z}} \sum_{b=0}^{\lfloor\alpha^2 \rfloor +1} \bigg( \sup_{y\in \left(\frac {b}a,\frac{b+1}a\right)}
\left|
\frac{\sin \pi a y}{\sin \pi (y+\frac qa)} 
\frac{\sin \pi a(\alpha^{-1}y+\frac {q\alpha^{-1}}a)}{\sin \pi (\alpha^{-1} y+\frac {q\alpha^{-1}}a)} 
\right| + 1 \bigg),
\end{align*}
which is an estimate that is uniform in $\beta\in[0,1)$. Now we again use the rational independence of $\{1,\alpha^{-1}\}$ and the fact that $|\sin \pi (y+\frac qa)|$ is periodic in $q\in \mathbb{Z}$ with period $a$ to gain (setting $\gamma= \{q\alpha^{-1}/a\}$) that 
\begin{equation}\label{eq:w2M}
M_2(2)-\delta' \le \lfloor\alpha^2 \rfloor +1 + \max_{q\in\{0,\ldots,a-1\}} \sup_{\gamma\in[0,1)}
\sum_{b=0}^{\lfloor\alpha^2 \rfloor +1}  \sup_{y\in \left(\frac {b}a,\frac{b+1}a\right)}
\left|
\frac{\sin \pi a y}{\sin \pi (y+\frac qa)} 
\frac{\sin \pi a(\alpha^{-1}y+\gamma)}{\sin \pi (\alpha^{-1} y+\gamma)} 
\right|.
\end{equation}
We have to derive upper bounds for the right hand side. Set 
\[
g(x) = \frac{\sin \pi a x}{\sin \pi x} 
\qquad
\hbox{and}
\qquad
h(y, \gamma,q) = g\Big(y+\frac qa\Big) g(\alpha^{-1} y+\gamma). 
\]
Then, Taylor expansion yields 
\[
  |h(y, \gamma,q)| \le |h(y_0,\gamma_0,q)| +
  \frac\varepsilon2\max_{(y',\gamma')\in
    J(\varepsilon,\eta)}|h_y(y',\gamma',q)| + \frac\eta2
  \max_{(y',\gamma')\in J(\varepsilon,\eta)}|h_\gamma(y',\gamma',q)|
\]
for
$(y,\gamma)\in J(\varepsilon,\eta):=(y_0-\varepsilon/2,
y_0+\varepsilon/2)\times(\gamma_0-\eta/2, \gamma_0+\eta/2)$; note that
these intervals depend on $y_0$ and $\gamma_0$. We now want to
estimate the derivatives. By the product rule,
\begin{align*}
|h_y(y, \gamma,q)| &\le  \Big|g'\Big(y+\frac qa\Big) g(\alpha^{-1} y+\gamma)\Big| + \alpha^{-1} \Big|g\Big(y+\frac qa\Big)g'(\alpha^{-1} y+\gamma)\Big|,
\\
|h_\gamma(y, \gamma,q)| &= \Big|g\Big(y+\frac qa\Big)g'(\alpha^{-1} y+\gamma)\Big|.
\end{align*}
Now, $|g(x)|\le a$ and by expanding $g$ in an exponential series we get
\[
|g'(x)| = \left| \sum_{j=0}^{a-1}2\pi \sqrt{-1} j e(j x)  \right| \le \pi a(a-1).
\]
Inserting this in \eqref{eq:w2M} yields that for each $\varepsilon, \eta > 0$
\[
\begin{split}
M_2(2) - \delta'& \le \lfloor\alpha^2 \rfloor +1  \\
&+\max_{q\in\{0,\ldots,a-1\}} 
\max_{\gamma_0\in \{\ell \eta:\ell\in\mathbb{N}\} \cap [0,1+\frac\eta2)}
\sum_{b=0}^{\lfloor\alpha^2 \rfloor +1}   \max_{y_0\in \{\ell \varepsilon:\ell\in\mathbb{N}\}  \cap \left(\frac {b}a
-\frac\varepsilon2 %%%
,\frac{b+1}a+\frac\varepsilon2\right)}
\\
&\hskip 1cm 
\Big\{|h(y_0,\gamma_0,q)| + \frac\varepsilon2\max_{(y',\gamma')\in J(\varepsilon,\eta)}|h_y(y',\gamma',q)| + \frac\eta2 \max_{(y',\gamma')\in J(\varepsilon,\eta)}|h_\gamma(y',\gamma',q)|\Big\},
\end{split}
\]
and thus, again for each $\varepsilon, \eta > 0$,
\begin{equation}\label{eq:forComputer}
\begin{split}
M_2(2)-\delta' \le &
\lfloor\alpha^2 \rfloor +1 \\
&+ \max_{q\in\{0,\ldots,a-1\}} 
\max_{\gamma_0\in \{\ell \eta:\ell\in\mathbb{N}\} \cap [0,1+\frac\eta2)}
\sum_{b=0}^{\lfloor\alpha^2 \rfloor +1}  
\max_{y_0\in \{\ell \varepsilon:\ell\in\mathbb{N}\}  \cap \left(\frac {b}a-\frac\varepsilon2,\frac{b+1}a+\frac\varepsilon2\right)} |h(y_0,\gamma_0,q)|\\
&+
\varepsilon a \max_{q\in\{0,\ldots,a-1\}} 
\sum_{b=0}^{\lfloor\alpha^2 \rfloor +1} 
\sup_{y\in \left(\frac {b}a-\frac\varepsilon2,\frac{b+1}a+\frac\varepsilon2\right)}
\Big|g'\Big(y+\frac qa\Big)\Big|\\
&+
\varepsilon\alpha^{-1} \pi a(a-1) \max_{q\in\{0,\ldots,a-1\}} 
\sum_{b=0}^{\lfloor\alpha^2 \rfloor +1} 
\sup_{y\in \left(\frac {b}a-\frac\varepsilon2,\frac{b+1}a+\frac\varepsilon2\right)}
\Big|g\Big(y+\frac qa\Big)\Big|\\
&+
\eta \pi a(a-1) \max_{q\in\{0,\ldots,a-1\}} 
\sum_{b=0}^{\lfloor\alpha^2 \rfloor +1} 
\sup_{y\in \left(\frac {b}a-\frac\varepsilon2,\frac{b+1}a+\frac\varepsilon2\right)}
\Big|g\Big(y+\frac qa\Big)\Big|.
\end{split}
\end{equation}
The estimation of $M_2(3)$ is much easier. By periodicity we have
\begin{equation}\label{eq:fc2}
M_2(3)- (\lfloor \alpha^3 \rfloor +2) \delta = \sup_{q\in\mathbb{Z}} \sum_{b=0}^{\lfloor\alpha^3 \rfloor +1} (M_2(3,b+q) - \delta)\le   \max_{q\in\{0,\ldots,a-1\}} \sum_{b=0}^{\lfloor\alpha^3 \rfloor +1}
\sup_{y\in \left(\frac ba,\frac{b+1}a\right)}
\Big|g\Big(y+\frac qa\Big)\Big|.
\end{equation}

%Although the suprema in \eqref{eq:forComputer}  and \eqref{eq:fc2} are extended over intervals it is easy to give good estimates for them. In each case we just subdivide the interval over which the supremum is extended in small subintervals. The supremum in each of these small intervals is estimated by maximizing the denominators, minimizing the numerators (these extrema often lie on the boundary of the interval in question) and bound $g$ and $g'$, respectively, by their trivial upper bounds if necessary (see Lemma~\ref{lem:mest} for a similar approach). 
Treating the estimates \eqref{eq:forComputer} and \eqref{eq:fc2} with {\tt Mathematica} (accelerated by a {\tt C} program for the calculation of the ``main term'' in the first line of \eqref{eq:forComputer}) and choosing $\delta=10^{-10}$ led to the results displayed in Table~\ref{tab:w2}. This is used to prove the following lemma.

\begin{table}
\begin{tabular}{||c|c|c|c|c|c||}
\hline 
\hline 
$a$ & $\varepsilon$ & $\eta$ & Upper bound for $M_2$ & Power of $\alpha$ & $\alpha^{3}$\\
\hline
\hline 
$39$ & 0.005 & 0.0005 &46695.7 & 2.93416 & 59436 \\ % OK EPS/2
\hline 
$38$ & 0.005 & 0.0005 &43255.2 & 2.93405 & 54986 \\ % OK EPS/2
\hline 
$37$ & 0.005 & 0.0005 & 39994.9& 2.93398 & 50764 \\ % OK EPS/2
\hline 
$36$ & 0.005 & 0.0005 & 36989.9& 2.93458 & 46764 \\ % OK EPS/2
\hline 
$35$ & 0.005 & 0.0008 &39595.4 & 2.97694 & 42980 \\ % OK EPS/2
\hline 
$34$ & 0.005 & 0.0008 &36279.6 & 2.97656 & 39406 \\ % OK EPS/2
\hline 
$33$ & 0.005 & 0.0008 & 33182.6 & 2.97641& 36036 \\ % OK EPS/2
\hline 
$32$ & 0.005 & 0.0008&30243.8 &2.97603 & 32864 \\ % OK EPS/2
\hline
$31$ & 0.005 & 0.0008 & 27544.8 &2.97627 & 29884  \\ % OK EPS/2
\hline
$30$ & 0.005 & 0.0008 & 24991.4 & 2.97630 & 27090  \\ %  CHANGED ($30$ & 0.005 & 0.0008 & 24991.2 & 2.97630 & 27090)
\hline
$29$ & 0.005 & 0.0008 & 22665.7 &{\bf 2.97719} & 24476  \\ % OK EPS/2
\hline
$28$ & 0.005 & 0.0007 &  19735.6& 2.96693& 22036  \\ % OK EPS/2
\hline
$27$ & 0.005 & 0.0007 & 17807.7 & 2.96839 & 19764  \\ % OK EPS/2
\hline
$26$ & 0.005 & 0.0007 & 16017.7 & 2.97016& 17654  \\ % OK EPS/2
\hline
$25$ & 0.005 & 0.0007 & 14374.2 &2.97261 & 15700  \\ % CHANGED ($25$ & 0.005 & 0.0007 & 14374.2 &2.97260 & 15700)
\hline
$24$ & 0.005 & 0.0007 &12841.2  &2.97517 & 13896  \\ % CHANGED ($24$ & 0.005 & 0.0007 &12841.1  &2.97517 & 1389)
\hline
$23$ & 0.005 & 0.0006 & 11122.8 & 2.96960 &  12236 \\  % OK EPS/2
\hline
$22$ &0.005 &0.0006& 9885.92 &2.97399 & 10714  \\ % OK EPS/2
\hline
$21$ &0.005 &0.0005 &8524.75  & 2.97059& 9324  \\  % OK EPS/2
\hline
$20$ &0.005 &0.0005 & 7518.04 & 2.97678& 8060   \\ % OK EPS/2
\hline
$19$ &0.005 & 0.0004& 6454.22 &2.97655 & 6916   \\ % OK EPS/2
\hline
$18$ &0.001 &0.0004 & 5303.48 &2.96398 & 5886  \\ % OK EPS/2
\hline
$17$ &0.001 &0.0004 & 4613.01 &2.97415 & 4964  \\ % OK EPS/2
\hline
$16$ &0.001 &0.0001 & 3773.67 &2.96628 & 4144  \\  % CHANGED ($16$ &0.001 &0.0001 & 3773.30 &2.96625 & 4144)  
\hline
$15$ &0.001 & 0.00003& 3212.43 & 2.97692 &3420 \\ % OK EPS/2
\hline 
\hline
\end{tabular}
\medskip
\caption{Results of the computer calculations for the upper bound of $M_{2}:=\max\{M_2(2),1\}+\max\{M_2(3),1\}^{2/3}$ in comparison with $\alpha^{3}$; see \eqref{eq:MwMw+1} for $w=2$.  The  entry in the column ``Power of $\alpha$'' is just the number $\kappa$ satisfying $M_2 \le  \alpha^\kappa$ according  to the estimate. \label{tab:w2}
}
\vskip -0.5cm
\end{table}

\begin{lem}\label{lem:estGspecial2}
Let $(G_j)_{j\ge0}$ be a linear recurrence base satisfying the conditions of Definition~\ref{def:lrb} whose characteristic polynomial is given by $X^2-aX-1$ and has dominant root $\alpha$. If $15\le a \le 39$ 
then
\[
  \int_0^1\lvert S_{n}(y,\beta)\rvert\mathrm{d}y \ll\alpha^{n\eta} \quad\hbox{ and }\qquad  \int_0^1\lvert \frac{\partial S_{n}}{\partial y}(y,\beta)\rvert\mathrm{d}y \ll \alpha^{n(1+\eta)}
\]
hold for some explicitly computable constant $\eta < 0.4886061$. 
\end{lem}

\begin{proof}
In view of \eqref{eq:intes}, \eqref{eq:newAbl} and \eqref{eq:MwMw+1} we have to show that 
\begin{equation}\label{eq:2223}
M_2:=\max\{M_2(2),1\}+\max\{M_2(3),1\}^{2/3} < \alpha^{1.4886061\cdot 2} =\alpha^{2.9772122}.
\end{equation}
This follows from the results listed in Table~\ref{tab:w2} (see the penultimate column whose largest value, which is typeset in boldface, is still smaller than $2.9772122$). 
\end{proof}

\medskip

We think that using blocks of length greater than two with increasing effort we can treat even smaller values of $a_1$. 

\section{Proofs of the main results}\label{sec:proof}

%Now we are in a position to finalize the proofs of Theorem~\ref{thm:bombieri-vinogradov-type} and Corollary~\ref{cor1}. 

\subsection{Proof of Theorem~\ref{thm:bombieri-vinogradov-type}}
In order to prove Theorem~\ref{thm:bombieri-vinogradov-type} we have to show that the estimate in \eqref{rewritten-estimate} holds. To this end we employ the exponential sum estimates established in Section~\ref{sec:estimates}. Moreover, we use the following inequality due to Sobolev and Gallagher (see \cite[Lemma 1.2]{Montgomery:71}).

\begin{lem}\label{mo:lem1.2}
  Let $T_0,T\geq\delta>0$ be real numbers and $f:[T_0,T_0+T]\to \C$ a continuously differentiable function. Furthermore let
  $\mathcal{R} \subset [T_0+\frac{\delta}{2},T_0+T-\frac{\delta}{2}]$ such that
  $\left| t-t'\right| \geq \delta$ holds for $t, t' \in \mathcal{R}$ with $t\not=t'$. Then we have the inequality
  \[\sum_{t\in\mathcal{R}}\lvert f(t)\rvert
    \leq \delta^{-1} \int_{T_0}^{T_0+T}\left| f(x)\right|\mathrm{d}x
    +\frac12\int_{T_0}^{T_0+T}\left| f'(x)\right|\mathrm{d}x.\]
\end{lem}

\begin{proof}[Proof of Theorem~\ref{thm:bombieri-vinogradov-type}]
We need to prove the estimate in \eqref{rewritten-estimate}. First we rewrite the sum on the right hand side of \eqref{rewritten-estimate} to get
  \begin{align}\label{eq:sumqdelta}
     \sum_{Q<q\leq 2Q}\sum_{h=1}^{q-1}\left| S_n\left(\frac hq,\frac{r}{s}\right)\right|
    =\sum_{\delta=1}^{2Q}\sum_{Q\delta^{-1}<q\leq 2Q\delta^{-1}}
      \sum_{\substack{h=1\\ (h,q)=1}}^{q-1}
      \left| S_{n}\left(\frac hq,\frac{r}{s}\right)\right|.
  \end{align}
Now we concentrate on the two innermost sums and set 
  \[
   L_Q(\delta):= \sum_{Q\delta^{-1}<q\leq 2Q\delta^{-1}}
      \sum_{\substack{h=1\\ (h,q)=1}}^{q-1}
      \left| S_{n}\left(\frac hq,\frac{r}{s}\right)\right|.
  \]
Since the estimate in \eqref{rewritten-estimate} is trivially true for $n\le D$ we will assume that $n > D$ in the sequel.
Using the product representation for $S_n$ in (\ref{eq:IprodbisB}) we obtain for each $n_1\in\{D,\ldots, n-1\}$ the estimate (we use the abbreviation $\mathbf{j}=(j_1,\ldots,j_k)$)
  \begin{equation*}
   L_Q(\delta)
    \leq\sum_{Q\delta^{-1}<q\leq 2Q\delta^{-1}}
    \sum_{\substack{h=1\\ (h,q)=1}}^{q-1}
    \sum_{k=1}^{n-n_1}\sum_{\mathbf{j}\in
    K_k(n_1)}\prod_{\ell=1}^k\lvert
    B_{n-\sum_{r=1}^{\ell-1}j_r,j_\ell}\left(\frac{h}{q},\frac{r}{s}\right)\rvert\cdot
    \lvert S_{n-\sum_{r=1}^k j_r}\left(\frac{h}{q},\frac{r}{s}\right)\rvert.
  \end{equation*}
Later we will choose $n_1$ depending on $Q$ and $\delta$.
By the definition of $K_k(n_1)$ in \eqref{eq:iterate2B} the index $n-\sum_{r=1}^k j_r$ always satisfies $n_1- D < n-\sum_{r=1}^k j_r \le n_1$. Thus 
\begin{align*}
   L_Q(\delta)
    &\leq\sum_{Q\delta^{-1}<q\leq 2Q\delta^{-1}}
    \sum_{\substack{h=1\\ (h,q)=1}}^{q-1}
     \max_{n_1-D<i\le n_1}\lvert S_{i}\left(\frac{h}{q},\frac{r}{s}\right)\rvert
    \sum_{k=1}^{n-n_1}\sum_{\mathbf{j}\in
    K_k(n_1)}\prod_{\ell=1}^k\lvert
    B_{n-\sum_{r=1}^{\ell-1}j_r,j_\ell}\left(\frac{h}{q},\frac{r}{s}\right)\rvert
    \\
   & \ll \sum_{n_1-D<i\le n_1} 
   \sum_{Q\delta^{-1}<q\leq 2Q\delta^{-1}}
    \sum_{\substack{h=1\\ (h,q)=1}}^{q-1}
     \lvert S_{i}\left(\frac{h}{q},\frac{r}{s}\right)\rvert
    \sum_{k=1}^{n-n_1}\sum_{\mathbf{j}\in
    K_k(n_1)}\prod_{\ell=1}^k\lvert
    B_{n-\sum_{r=1}^{\ell-1}j_r,j_\ell}\left(\frac{h}{q},\frac{r}{s}\right)\rvert .
  \end{align*}
Now we apply Proposition~\ref{prop:S-maxA} which yields
\begin{equation*}
   L_Q(\delta)
     \ll   \sum_{n_1-D<i\le n_1} \alpha^{\lambda(n-n_1)}\sum_{Q\delta^{-1}<q\leq 2Q\delta^{-1}}
    \sum_{\substack{h=1\\ (h,q)=1}}^{q-1}
   \lvert S_{i}\left(\frac{h}{q},\frac{r}{s}\right)\rvert
  \end{equation*}
for some constant $\lambda<1$. In this estimate $\lambda$ and the implied constant depend only on $G$ and $s$. In the next step we apply Lemma \ref{mo:lem1.2} together with the 1-norm estimates in Propositions~\ref{prop:S-1norm} and~\ref{prop:S'-1norm2}. Setting $\eta= \log_\alpha \min\{\alpha^{\frac12},(m+3)\} \le \frac12$ we get
\begin{align}
L_Q(\delta)
&
\ll \sum_{n_1-d<i\le n_1}  \alpha^{\lambda(n-n_1)} 
\bigg(Q^2\delta^{-2}
\lVert
      S_{i}\left(\cdot,\frac{r}{s}\right)
       \rVert_1+
    \lVert  
     \frac{\partial S_{i}}{\partial y}\left(\cdot,\frac{r}{s}\right)
     \rVert_1 \bigg) \nonumber
  \\   
&\ll \alpha^{\lambda(n-n_1)}  (Q^2\delta^{-2}\alpha^{\eta n_1} +  \alpha^{(1+\eta)n_1}). \label{eq:twosummands}
 \end{align} 
We choose $n_1$ by setting
\[
n_1:=\min\left(\lfloor2\log_\alpha(Q\delta^{-1})\rfloor + D,n-1\right).
\]
we gain (note that for $n_1=\lfloor2\log_\alpha(Q\delta^{-1})\rfloor+D$ both summands in \eqref{eq:twosummands} are roughly of the same size)
\[
 L_Q(\delta) \ll Q^2\delta^{-2}\alpha^{\eta n} + \alpha^{\lambda n}\alpha^{2(1+\eta-\lambda)\log_\alpha(Q\delta^{-1})}
 = Q^2\delta^{-2}\alpha^{\eta n}+\alpha^{\lambda n}(Q\delta^{-1})^{2(1+\eta-\lambda)}.
\]
It suffices to prove the theorem for small $\varepsilon$. 
Thus if $\eta  < \frac12$ we may assume that $2\eta + \varepsilon <1$.
On top of this, for all $\eta \le \frac12$ we may assume that 
$\varepsilon$ is small enough and that 
the constant  $\lambda<1$ from Proposition~\ref{prop:S-maxA} is close enough to $1$ such that 
$\varepsilon(\frac12 -\varepsilon) < 1-\lambda < \frac{\varepsilon}2$ 
holds (note that if we increase $\lambda$, the estimate in Proposition~\ref{prop:S-maxA} clearly remains valid). This yields
\[
 L_Q(\delta) \ll Q^2\delta^{-2}\alpha^{\eta n} + \alpha^{\lambda n} (Q\delta^{-1})^{2\eta+\varepsilon}
 %\ll Q^2\delta^{-2}\alpha^{\eta n} + \alpha^{\lambda n} (Q\delta^{-1})^{1+\varepsilon}.
\]
Taking into account the sum over $\delta$ in \eqref{eq:sumqdelta} we end up with
 \begin{equation}\label{eq:finalinsert}
 \sum_{Q<q\leq 2Q}\sum_{h=1}^{q-1}\left| S_n\left(\frac hq,\frac{r}{s}\right)\right|
 \ll
 \begin{cases}
 Q^2\alpha^{\eta n} +  \alpha^{\lambda n} Q^{1+\varepsilon} & \hbox{if } \eta = \frac 12,\\
 Q^2\alpha^{\eta n} +  \alpha^{\lambda n} Q& \hbox{if } \eta < \frac 12.
 \end{cases}
 \end{equation}
Let $\vartheta = 1-\eta$. Then $\vartheta\ge \frac 12$ and by Lemma~\ref{lem:mest} we have  $\vartheta\to1$ for $a_1\to \infty$. 
Also recall that  $Q \le x^{\vartheta-\varepsilon}$ and $n\le \log_\alpha x+C$ for some constant $C$ depending on $G$. Thus for $\eta=\frac12$ we get 
\begin{equation}\label{etacase12}
Q^2\alpha^{\eta n} +  \alpha^{\lambda n} Q^{1+\varepsilon} \ll Qx^{\eta +\vartheta - \varepsilon} + Qx^{\lambda + \varepsilon(\vartheta-\varepsilon)}
= Qx^{1 - \varepsilon} + Qx^{\lambda + \varepsilon(\frac{1}{2}-\varepsilon)} \ll Qx^{\gamma}
\end{equation}
for some $\gamma<1$. For $\eta< \frac12$ we gain
\begin{equation}\label{etacaseless12}
Q^2\alpha^{\eta n} +  \alpha^{\lambda n} Q \ll Qx^{\eta + \vartheta - \varepsilon} + Qx^{\lambda}= Qx^{1 - \varepsilon} + Qx^{\lambda} \ll Qx^{\lambda}.
\end{equation}
Inserting \eqref{etacase12} and \eqref{etacaseless12} in \eqref{eq:finalinsert} we finally see that
\[
 \sum_{Q<q\leq 2Q}\sum_{h=1}^{q-1}\left| S_n\left(\frac hq,\frac{r}{s}\right)\right|
 \ll
Qx(\log 2x)^{-A}
\]
holds for each $A >0$ and the proof is finished.
\end{proof} 

\subsection{Improvements on the level of distribution  and proof of Corollary~\ref{cor1}}
Let $G$ be a linear recurrence base as in Definition~\ref{def:lrb} and let $\alpha$ be the dominant root of the characteristic polynomial $X^d-a_1X^{d-1}-\cdots- a_{d-1}X - a_d$ of $G$. In the proof of Theorem~\ref{thm:bombieri-vinogradov-type} we see that the level of distribution $\vartheta(G)$ is equal to $1-\eta$ where $\eta$ satisfies $\Vert S_n(\cdot,\beta)\Vert_1 \ll \alpha^{\eta n}$ and $\Vert \frac{\partial S_n}{\partial y}(\cdot,\beta)\Vert_1 \ll \alpha^{(1+\eta)n}$. Together with our estimates of these 1-norms, we gain the following result.

\begin{lem}\label{lem:ThetaImproved}
  Let $G=(G_j)_{j\ge0}$ be a linear recurrence base whose
  characteristic polynomial is given by $X^d-a_1X^{d-1}-\cdots- a_{d-1}X-a_d$.  If $a_1 \ge 59$
  then in Theorem~\ref{thm:bombieri-vinogradov-type} the level of distribution satisfies 
  \[
  \vartheta(G) \ge  0.5113939 = 1-0.4886061. 
  \]
  If the characteristic polynomial of $G$ is of the special form $X^2-a_1X-1$ then this estimate even holds for $a_1 \ge 15$.
 \end{lem}
 
 \begin{proof}
 From Lemma~\ref{lem:estallgG} we see that $\Vert S(\cdot,\beta)\Vert \ll \alpha^{0.4886061}$ for $a_1\ge 59$. This proves the first assertion. 
 
 If the characteristic polynomial of $G$ is of the special form $X^2-a_1X-1$ then for $a_1 \ge 40$ the result follows because Lemma~\ref{lem:estGspecial} yields again $\Vert S(\cdot,\beta)\Vert \ll \alpha^{0.4886061}$. If $15 \le a_1 \le 39$ then the result is a consequence of Lemma~\ref{lem:estGspecial2}. 
 \end{proof}

Along the lines indicated in Section~\ref{sec:13} we can now prove Corollary~\ref{cor1}.

\begin{proof}[Proof of Corollary~\ref{cor1}]
From Greaves~\cite[Proposition~1 (see also Theorem~1) of Chapter~5]{Greaves:01}) it follows that \eqref{eq:cor1} holds provided that $\frac {1}{\vartheta(G)} < 2 - \delta_2$ for a certain constant $\delta_2$. Since $\delta_2=0.044560$ is an admissible choice for this constant according to Greaves~\cite{Greaves:86}, we conclude that  \eqref{eq:cor1} holds if $\vartheta(G) > 0.5113938\ldots$ Since this is true in view of Lemma~\ref{lem:ThetaImproved} whenever the conditions of the corollary are in force, the result is established.
\end{proof}

%\section{Perspectives and problems}
%
%\begin{itemize}
%\item Fibonacci $1$-norm
%\item Primes, Mauduit-Rivat; analogs of Gelfond's conjectures.
%\item Fixed sum-of-digits
%\item Drmota's conjecture for Fibonacci numbers and primes
%\end{itemize}

\section*{Acknowledgement}

Major parts of the present paper were established when the first
author was visiting the Chair of Mathematics and Statistics at the
University of Leoben, Austria. He thanks the institution for
its hospitality.

\bibliographystyle{abbrv}
\bibliography{bv-linear-recurrent}

\end{document}